\documentclass{amsart}
\usepackage[T1]{fontenc}							
\usepackage[english]{babel}							
\usepackage[utf8]{inputenc}							
\usepackage{amsmath,amsthm, amsfonts,amsrefs,amssymb, enumerate}
\usepackage{mathrsfs}
\usepackage{mathtools}
\mathtoolsset{mathic} 
\usepackage{tikz}

\usepackage{bm} 
\usepackage{bbm} 
\usepackage{mathrsfs} 
\usepackage{dsfont}
\usepackage{enumitem}
\usepackage{bbm}

\newtheorem{theorem}{Theorem}[section]
\newtheorem{lemma}[theorem]{Lemma} 
 
\newtheorem{proposition}[theorem]{Proposition} 
 
\theoremstyle{definition}
\newtheorem{definition}[theorem]{Definition}
\newtheorem{remark}[theorem]{Remark} 

\definecolor{pink}{rgb}{1,0,1}
\def\beq{\begin{eqnarray}}
\def\eeq{\end{eqnarray}}
\newcommand{\nn}{\nonumber}

\newcommand*{\GL}{\text{GL}}

\newcommand*{\T}{\mathbb{T}}
\newcommand*{\R}{\mathbb{R}}
\newcommand*{\N}{\mathbb{N}}
\newcommand*{\C}{\mathbb{C}}
\newcommand*{\Z}{\mathbb{Z}}

\newcommand{\pa}{\partial}
\usepackage{graphicx}								
\usepackage{subfig}									
\usepackage{listings}								
\usepackage{eso-pic}								

\usepackage{float} 									

\usepackage[labelfont=bf, textfont=normal,
			justification=justified,
			singlelinecheck=false]{caption}

\usepackage{hyperref}								
\hypersetup{colorlinks, citecolor=black,
   		 	filecolor=black, linkcolor=black,
    		urlcolor=black}

\usepackage[textsize=tiny]{todonotes}   

\usepackage{longtable}

\begin{document}


\title{112 years of listening to Riemannian manifolds}
\author[G.\@ M\aa rdby]{Gustav M\aa rdby}
\address{Mathematical Sciences, Chalmers University of Technology and University
of Gothenburg, 412 96 Gothenburg, Sweden}
\email{mardby@chalmers.se}

\author[J.\@ Rowlett]{Julie Rowlett}
\address{Mathematical Sciences, Chalmers University of Technology and University
of Gothenburg, 412 96 Gothenburg, Sweden}
\email{julie.rowlett@chalmers.se}

\maketitle

\begin{abstract}
In 1910, Hendrik Antoon Lorentz delved into the enigmatic Laplace eigenvalue equation, also known as the Helmholtz equation, pondering to what extent the geometry in which one solves the equation can be recovered from knowledge of the eigenvalues. Lorentz, inspired by physical and musical analogies, conjectured a fundamental relationship between eigenvalues, domain volume, and dimensionality. While his conjecture initially seemed insurmountable, Hermann Weyl's groundbreaking proof in 1912 illuminated the deep connection between eigenvalues and geometric properties. Over the ensuing 112 years, mathematicians and physicists have continued to decipher the intricate interplay between eigenvalues and geometry. From Weyl's law to Milnor's example of isospectral non-isometric flat tori, and Kac's inspiring question about hearing the shape of a drum, the field has witnessed remarkable progress, uncovering spectral invariants and advancing our understanding of geometric properties discernible through eigenvalues.  We present an overview of this field amenable to both physicists and mathematicians.
\end{abstract}

\section{Introduction} \label{sec:introduction}
In 1910 the Dutch physicist Hendrik Antoon Lorentz gave five lectures under the title ``Old and new problems in physics'' at the University of Göttingen.  One of the central topics of his lectures is the Helmholtz equation, the eigenvalue equation for the Laplace operator:  \begin{equation} \label{eq:eigenvalue_problem}
    \Delta f = \lambda f, \quad \textrm{together with boundary conditions in case of boundary.}
\end{equation}
Here, $\Delta$ is the Laplace operator, which in $n$-dimensional Euclidean space $\R^n$ is 
\begin{equation} \label{eq:Laplace_R^n}
    \Delta = -\sum_{i=1}^n \frac{\partial^2}{\partial x_i^2}.
\end{equation}
Solving the Helmholtz equation amounts to finding all eigenvalues $\lambda$ such that there exists a solution $f$ to \eqref{eq:eigenvalue_problem}.  When one solves this problem on a bounded domain in Euclidean space, a boundary condition must also be imposed.  Perhaps the most common is the Dirichlet boundary condition, that requires the solution $f$ to vanish on the boundary of the domain.  At the time of Lorentz's lectures, experts could solve this equation in only a few types of domains:  boxes (products of intervals), spheres, disks, wedges, and products of these.  For a more general bounded domain in Euclidean space, the form of the solution and the collection of eigenvalues was still a mystery.       

Nonetheless, using reasoning from physics, Lorentz made a famous conjecture at the end of his fourth lecture:  
\beq \lim_{\Lambda \to \infty} \frac{ \# \{ \lambda \leq \Lambda\} }{|\Omega| \Lambda^{n/2}} = C > 0. \label{eq:lorentzs_con} 
\eeq 
In this equation, $\# \{ \lambda \leq \Lambda\}$ is the number of Laplace eigenvalues less than or equal to $\Lambda$, considered in a bounded domain $\Omega \subset \R^n$ taking the Dirichlet boundary condition.  Here, $|\Omega|$ is the $n$-dimensional volume of $\Omega$, and $C$ is a constant that depends only on the dimension $n$.  Lorentz had the following physical - and musical - explanation for this conjecture. Consider an enclosure with a perfectly reflecting surface, which produces electromagnetic waves analogous to tones of an organ pipe. To compute the energy in a frequency interval $dv$, one may calculate the number of overtones which lie between the frequencies $v$ and $v+dv$ and then multiply this by the energy at the frequency $v$. 
Now, if the conjecture is true, then the number of sufficiently high overtones between $v$ and $v+dv$ should be independent of the shape of the enclosure and only proportional to its volume \cite[p. 4]{kac1966can}.

Lorentz could rigorously prove that the conjecture is true if the domain were an $n$-dimensional Euclidean box. So, he reasoned based on the physical interpretation  that the same asymptotic relationship should hold for any bounded domain. This problem had already been formulated a month eariler by Arnold Sommerfeld at a lecture in Köningsberg. David Hilbert, who was present at Lorentz's lecture, was convinced that \eqref{eq:lorentzs_con} would not be proven in his lifetime. However, using techniques developed by Hilbert himself, Hermann Weyl, who was also present at the lecture and was a student of Hilbert, proved Lorentz's conjecture less than two years later in 1912 \cite{weyl1912asymptotische}!  Thus, Weyl proved that the eigenvalues of the Laplace operator acting on a bounded Euclidean domain with the Dirichlet boundary condition detect both the  dimension $n$ as well as the $n$-dimensional volume of the domain. Inspired by Lorentz's organ pipe, one could say that these eigenvalues can be \em heard, \em and describe their study as \em listening \em to them.  In this way we could also describe those quantities that are specified by knowing the eigenvalues can also be \em heard.  \em  The title of this article refers to 112 years of investigating the interplay between the eigenvalues and the geometric setting in which one is solving the Helmholtz equation, because one can in general study the Helmholtz equation on an $n$-dimensional Riemannian manifold $(M,g)$.  

In this case, the Laplace operator 
is 
\begin{equation} \label{eq:Laplace_(M,g)}
    \Delta = -\sum_{i,j} \frac{1}{\det(g)}\partial_i g^{ij}\sqrt{\det(g)}\partial_j.
\end{equation}

To further connect the Helmholtz equation to our listening analogy, note that sounds are caused by vibration which is mathematically described by the wave equation. A common way to solve the wave equation is to use separation of the time and space variables. When one does this, the equation for the space variables turns into the eigenvalue problem \eqref{eq:eigenvalue_problem}.  The spectrum of $(M,g)$ is the set of all eigenvalues.  From a physical point of view, the spectrum represents the frequencies that $M$ can produce. Therefore, it is indeed reasonable to say that a quantity may be “heard” if it is determined by the eigenvalues. In other words, the quantity can be heard precisely if it is a spectral invariant.  In the 112 years since Weyl's discovery, mathematicians and physicists have investigated what geometric properties can be `heard.'  In other words, we have been listening to Riemannian manifolds to understand to what extent their geometry, or more colloquially, their \em shape \em can be \em heard. \em  We do not take credit for this language, as it was inspired by Lorentz and later popularized by Kac \cite{kac1966can}.  Interestingly, Kac's description and language came from Bochner, for whom the Bochner Laplacian is named  \cite{morin2020spectral}. 

The intended audience of this article ranges from experts to those who are new to the field. For the beginners, we will give an introduction to this rich field of mathematics. For the experts, we will give a survey which presents numerous works from different niches of the field, some of which even experts might not be aware. Therefore, not only do we give a thorough overview of the major results of the past 112 years, but we also convey essential techniques by explaining the proofs of some of these results. We start in \S \ref{s:weyl_law} with the proof of Lorentz's conjecture \eqref{eq:lorentzs_con}, a fact that is now known as Weyl's law.  We show how one can prove this fact using variational principles and bracketing arguments, as Weyl did 112 years ago.  Although several other proofs are now available \cite{ivrii1981asymptotics}, \cite{ivrii2016100}, \cite{canzani_galkowski2023}, \cite{canzani_galkowski_beams}, Weyl's original proof remains instructive and robustly applicable to settings in which similar variational principles are known to hold. 

Weyl's law shows that the dimension and the volume are spectral invariants, in other words one can hear the dimension and volume of a bounded domain.  It took over 40 years to discover the next geometric spectral invariant.  In 1954 the Swedish mathematician Åke Pleijel showed that one can hear the $(n-1)$-dimensional volume of the boundary of such a domain \cite{pleijel1954study}.  Of course, the field was not dormant during these decades, so what were people working on, and what results were they obtaining?  Since one cannot compute the eigenvalues of arbitrarily shaped domains, it is natural to turn to more computationally accessible examples.  One large class of examples is flat tori, compact smooth Riemannian manifolds obtained by taking the quotient of $\R^n$ by a full-rank lattice with Riemannian metric inherited from the flat Euclidean metric.  For these examples, one can show that the eigenvalues are equal to the representation numbers of certain naturally associated quadratic forms.  In the years from 1912 to 1954, the theory of quadratic forms and representation numbers developed rapidly and therewith our understanding of the spectra of flat tori.  In \S \ref{sec:quadratic_forms} we explain these connections and how they were used to construct Milnor's example of 16 dimensional isospectral non-isometric flat tori.  Milnor's example helped to inspire Kac's famous article \textit{Can one hear the shape of a drum?}


In the aforementioned article Kac focused on the Helmholtz equation in the special context of bounded, two-dimensional domains, so these domains could be seen as the heads of drums.  Then, one can ask, if these drums \em sound identical \em in the mathematical sense that their Laplace spectra are identical, then are the drums also the same shape?  In other words, is the entire geometry of a two-dimensional domain a spectral invariant?  Although Kac did not answer this question, he gave heuristic physical arguments to show that one would expect to be able to hear the Euler characteristic of a smoothly bounded domain.  Kac was forthright in explaining that his arguments were not fully rigorous.  Interestingly, his intuition was entirely correct, and we will show in \S \ref{sec:kac} that all of Kac’s arguments can be made rigorous.  Kac's paper has inspired a wealth of research in the field, an overview of which we give in \S \ref{sec:Kac_legacy} and \S\ref{sec:Kac_positive}. In \S \ref{sec:outlook}, we leave the reader with a brief look towards the future of the field and what one may expect or hope can  be achieved. 

\section*{Acknowledgements}
 We are grateful to Jeff Galkowski and Javier Gomez-Serrano for suggesting further references to include in this work.  
 

\section{Weyl's asymptotic law} \label{s:weyl_law}
In 1910, Weyl was a relatively young mathematician who had recently completed his doctoral studies. He earned his Ph.D. in mathematics in 1908 at the University of Göttingen under the supervision of David Hilbert, whom he greatly admired. Weyl's career was just beginning to take shape at that time, and he would go on to make significant contributions to various fields in mathematics and theoretical physics in the following decades, becoming a renowned figure in the world of mathematics. 

Weyl was very interested in mathematical physics at this point of time, particularly in the study of quantum mechanics and the behavior of quantum systems. When Weyl heard Lorentz mention the problem
about the volume being a spectral invariant, it captured his interest as he quickly realized that it was essential for understanding the behavior of quantum mechanical systems. He aimed to establish this law as a fundamental mathematical principle that could provide insights into the spectral behavior of quantum systems, helping to bride the gap between mathematics and theoretical physics. This was a groundbreaking and evolving field of study during this time \cite{newman1957hermann}.

\subsection{Preliminaries} 
Before we show how Weyl eventually solved the problem formulated by Lorentz, we establish key notations and preliminaries. A \textit{domain} $\Omega \subset \R^n$ is a connected open subset of $\R^n$. We always assume that $\overline{\Omega}$ is compact. The boundary of $\Omega$ is denoted by $\partial\Omega$. For a domain $\Omega \subset \R^n$, we let $|\Omega|$ denote the $n$-dimensional volume of $\Omega$ and $|\partial\Omega|$ the $(n-1)$-dimensional volume of the boundary of $\Omega$. In particular, if $n = 2$ then $|\Omega|$ and $|\partial\Omega|$ denote the area and the perimeter of $\Omega$, respectively.

The \textit{Laplace operator} on $\R^n$ is given by \eqref{eq:Laplace_R^n}.  For a function $f : \Omega \to \C$, we write $f \in \mathscr{C}^\infty(\Omega)$ if $f$ is infinitely differentiable, that is, if $f$ is $N$ times differentiable for every $N \in \N$. Moreover, we write $f \in \mathscr{C}_c^\infty(\Omega)$ if $f \in \mathscr{C}^\infty(\Omega)$ and the support of $f$ is compactly contained in $\Omega$. 

A function $f \in \mathscr{C}^\infty(\Omega)$ is said to satisfy \textit{Dirichlet boundary conditions} if $f = 0$ on $\partial\Omega$. Similarly, $f$ satisfies \textit{Neumann boundary conditions} if
    \begin{equation*}
        \frac{\partial f}{\partial n} = 0 \text{ on } \partial\Omega,
    \end{equation*}
where $\frac{\partial f}{\partial n}$ is the normal derivative of $f$ pointing outwards from $\Omega$. If $f$ is not identically zero, satisfies Dirichlet boundary conditions, and \eqref{eq:eigenvalue_problem} holds for some $\lambda \in \C$, then $\lambda$ is an \textit{eigenvalue} with corresponding \textit{eigenfunction} $f$. The \textit{Dirichlet/Neumann spectrum} of $\Omega$ is the set of all eigenvalues with eigenfunctions satisfying the Dirichlet/Neumann boundary condition, respectively.

For one dimensional domains, it is perfectly sufficient to solve the Helmholz equation by searching for smooth functions that satisfy either the Dirichlet or Neumann boundary conditions.  However, for domains in higher dimensions, especially in cases when the boundary of the domain is not smooth, to rigorously define the Laplace operator, one must specify the Hilbert space of functions on which the operator acts and is essentially self-adjoint. For this reason and to be inclusive of all readers we recall

\begin{definition} \label{def:L2_Hk}
For a domain $\Omega \subset \R^n$, $L^2(\Omega)$ is defined as the set of equivalence classes consisting of Lebesgue measurable functions $f : \Omega \to \C$ such that the Lebesgue integral $\int_\Omega |f(x)|^2dx < \infty$, where two functions $f,g$ are equivalent if
\begin{equation*}
\{x \in \Omega : f(x) \neq g(x)\}
\end{equation*}
has Lebesgue measure 0. We write $f \in L^2(\Omega)$ to indicate that there exists an equivalence class in $L^2(\Omega)$ containing $f$. If $f \in L^2(\Omega)$ and in addition the weak partial derivatives of $f$ of order up to $k$ exist and are in $ L^2(\Omega)$, then we write $f \in H^k(\Omega)$. For further details of these Hilbert spaces we refer to \cite[p. 51]{folland1999real}. The space $H^1 _0(\Omega)$ is the completion of $\mathscr C^\infty _c (\Omega) $ with respect to the norm of the Hilbert space $H^1 (\Omega)$ (which properly contains $\mathscr C^\infty _c (\Omega)$). 
\end{definition}

Now let $(M,g)$ be a compact Riemannian manifold, possibly with boundary.  The Laplace operator is defined in \eqref{eq:Laplace_(M,g)}.  If $\pa M = \emptyset$ then the operator acts on $ H^2(M)$, or if $\pa M \neq \emptyset$ and the Laplace operator acts on $H^2(M)$ we refer to this as the Neumann Laplace operator.  The Dirichlet Laplace operator acts on $ H^1 _0 (M) \cap H^2(M)$.  These definitions are completely analogous for (bounded) domains in $\R^n$.  In all of the aforementioned cases, the eigenvalues are discrete and non-negative, forming a countable set in $[0, \infty)$.   We let $N(\lambda)$ denote the number of eigenvalues less than or equal to $\lambda$, 
    \begin{equation} \label{eq:N_lambda}
        N(\lambda) = \#\{n \in \N : \lambda_n \leq \lambda\}.
    \end{equation} 
Here we count multiplicity, so that an eigenvalue with, say, $m$ linearly independent eigenfunctions is counted $m$ times. 

To state Weyl's law, we recall the notion of asymptotic equivalence of functions.  Two functions $f,g : \R \to \C$ are \textit{asymptotically equal} as $x \to \infty$ if
    \begin{equation*}
        \lim_{x \to \infty} \frac{f(x)}{g(x)} = 1.
    \end{equation*}
We write this as $f(x) \sim g(x), \,\,x \to \infty$. Functions being asymptotically equal as $x \to 0$ are defined analogously.

\subsection{Weyl's law} \label{sec:Weyl_approx}
Here we show how Weyl proved what is now known as \textit{Weyl's law}. In two dimensions, Weyl's law says that
\begin{equation} \label{eq:Weyl}
    \lim_{\lambda \to \infty} \frac{N(\lambda)}{\lambda} = \frac{|\Omega|}{4\pi}.
\end{equation}

One key idea needed to prove \eqref{eq:Weyl} is to cover $\Omega$ by rectangles and consider the eigenvalues on those. This is useful because on rectangles the eigenvalues and eigenfunctions can be computed analytically. Indeed, suppose

$\Omega = (0, l_1) \times (0, l_2)$ for some $l_1, l_2 > 0$ as shown in Figure \ref{fig:rectangle_domain}. Consider the eigenvalue problem with Dirichlet boundary conditions: 
\begin{equation} \label{eq:eigenvalue_rectangle}
        \begin{cases}
            \Delta U(x,y) = \lambda U(x,y) \text{ in } (0, l_1) \times (0, l_2), \\
            U(x,y) = 0 \text{ on } \partial\Omega.
        \end{cases}
\end{equation}
\begin{figure} 
    \centering
    \begin{tikzpicture}[x=1cm,y=0.4cm]
    \draw[<->] (-1,0)--(7,0); 
    \draw[<->] (0,-2)--(0,8); 
    \draw[-] (5,0)--(5,6); 
    \draw[-] (0,6)--(5,6); 
    \draw (5,0) circle[radius=2pt]; and \fill (5,0) node[below = 2pt]{$(l_1,0)$} circle[radius=2pt];
    \draw (0,6) circle[radius=2pt]; and \fill (0,6) node[left = 2pt]{$(0,l_2)$} circle[radius=2pt];
    \draw (5,6) circle[radius=2pt]; and \fill (5,6) node[above = 2pt]{$(l_1,l_2)$} circle[radius=2pt];
    \draw (2.5,2) circle[radius=0pt]; and \fill (2.5,2) node[above = 2pt]{$\Omega$} circle[radius=0pt];
\end{tikzpicture}
    \caption{This domain $\Omega$ is a rectangle with side lengths $l_1$ and $l_2$.}
    \label{fig:rectangle_domain}
\end{figure}
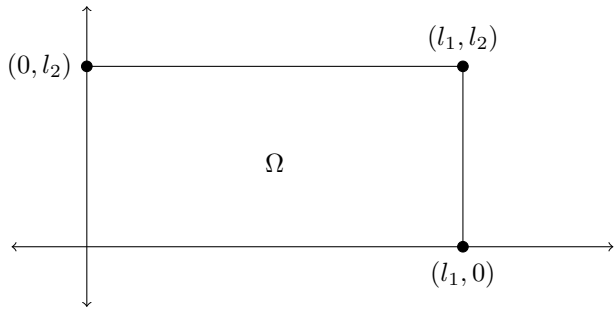
We solve this by separating variables. In other words, we look for solutions of the form $X(x)Y(y)$. Inserting this into \eqref{eq:eigenvalue_rectangle} gives
\begin{equation*} 
        X''(x)Y(y) + X(x)Y''(y) + \lambda X(x)Y(y) = 0
\end{equation*}
or
\begin{equation*}
    \frac{X''(x)}{X(x)} = -\frac{Y''(y)}{Y(y)} - \lambda.
\end{equation*}
The left side depends only on $x$ whereas the right side depends only on $y$ therefore both sides are constant.  They are essentially the same equation, which is a regular Sturm-Liouville problem, thanks to the boundary condition.  For $X$ this can be written, for some constant $\mu$, as 
\[ X''(x) = \mu X(x), \quad X(0) = 0, \quad X(\ell_1) = 0.\]
If $\mu = 0$, then the equation becomes $X''(x) = 0$, whose general solution is $X(x) = Ax + B$. The boundary conditions immediately give that $A = B = 0$, so $X(x) = 0$ is the only solution in this case. If $\mu \neq 0$, then all solutions are linear combinations of $e^{x \pm \sqrt \mu }.$  It is straightforward to show that the only non-zero linear combination that satisfies the boundary condition is a scalar multiple of $\sin(m \pi x/\ell_1)$ with corresponding $\mu_m = - m^2 \pi^2/\ell_1^2$ for $n \in \Z \setminus \{0\}$.  Since $\sin(-x)$ and $\sin(x)$ are linearly dependent, we simply set $X_m(x) = \sin(m\pi x/\ell_1)$ for $m \in \N$. Proceeding analogously we find $Y_n (y) = \sin(n \pi x / \ell_2)$ for $n \in \N$.  The eigenvalues and corresponding eigenfunctions for the rectangle are then respectively 
\beq \lambda_{m,n} = \frac{m^2 \pi^2}{\ell_1^2} + \frac{n^2 \pi^2}{\ell_2^2}, \quad \sin \left( \frac{m \pi x}{\ell_1} \right) \sin \left( \frac{n \pi y}{\ell_2} \right), \quad m, n \in \N.  \label{eq:d_ev_rect} \eeq
By \cite[p. 83]{larsson2003partial} these eigenfunctions constitute an orthogonal basis of $L^2((0, l_1) \times (0, l_2))$, and therefore we have found all eigenfunctions (uniquely up to multiplication by non-zero scalars) and therewith all eigenvalues. 

Now assume that \eqref{eq:eigenvalue_rectangle} has the Neumann boundary conditions. Then we solve the equation in the same way 
and obtain eigenvalues and corresponding eigenfunctions 
\beq \lambda_{m,n} = \frac{m^2 \pi^2}{\ell_1^2} + \frac{n^2 \pi^2}{\ell_2^2}, \quad \cos \left( \frac{m \pi x}{\ell_1} \right) \cos \left( \frac{n \pi y}{\ell_2} \right), \quad m, n \in \Z_{\geq 0}.  \label{eq:n_ev_rect} \eeq
Similarly, in this case these eigenfunctions constitute an orthogonal basis of $L^2((0, l_1) \times (0, l_2))$, and therefore we have found all eigenfunctions (uniquely up to multiplication by non-zero scalars) and therewith all eigenvalues.  Knowing the eigenvalues and eigenfunctions on rectangles, we can prove that Weyl's law holds on them for both the Dirichlet and Neumann boundary conditions. 

\begin{proposition} \label{prop:Weyl_rectangles}
    Let $I = (a_1, b_1) \times (a_2, b_2) \subset \mathbb{R}^2$ be a rectangle. Then both the Dirichlet and Neumann eigenvalues on $I$ satisfy 
    \begin{equation*}
        \lim_{\lambda \to \infty} \frac{N(\lambda)}{\lambda} = \frac{|I|}{4\pi}.
    \end{equation*}
\end{proposition}
\begin{proof}
By possibly translating $I$, we may assume that $a_i = 0$ and $b_i = l_i$, for $i = 1, 2$.  The Dirichlet eigenvalues are given in \eqref{eq:d_ev_rect} while the Neumann eigenvalues are given in \eqref{eq:n_ev_rect}.  For any $R > 0$, note that $N(R^2)$ is the number of lattice points $p = (m, n)$ such that $(m\pi/l_1)^2 + (n\pi/l_2)^2 \leq R^2$. This is asymptotically equal to the area of the quarter ellipse
    \begin{equation*}
        E_R := \left\{(x,y) : \frac{x^2\pi^2}{l_1^2} + \frac{y^2\pi^2}{l_2^2} \leq R^2, \,\, x,y \geq 0 \right\},
    \end{equation*}
    namely $|E_R| = l_1l_2R^2/4\pi$.

    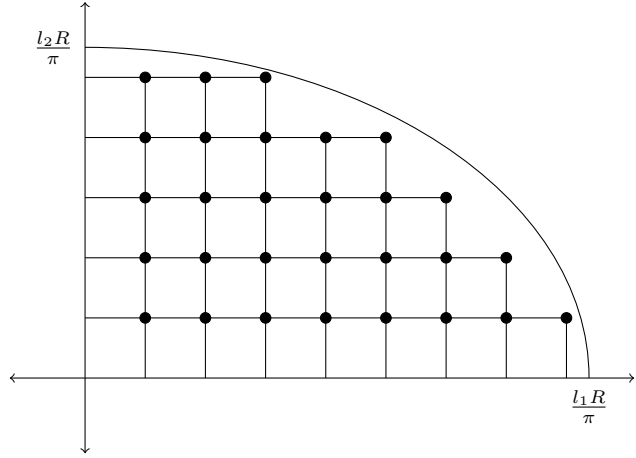
\begin{figure}
    \centering
    \begin{tikzpicture}
        \draw[<->] (-1,0)--(7.3,0); 
        \draw[<->] (0,-1)--(0,5); 
        \draw (6.7,0) arc
        [
            start angle=0,
            end angle=90,
            x radius=6.7cm,
            y radius = 4.4cm
        ] ;
        \draw (0.8,0.8) circle[radius=2pt]; and \fill (0.8,0.8) node[above = 2pt]{} circle[radius=2pt];
        \draw (1.6,0.8) circle[radius=2pt]; and \fill (1.6,0.8) node[above = 2pt]{} circle[radius=2pt];
        \draw (2.4,0.8) circle[radius=2pt]; and \fill (2.4,0.8) node[above = 2pt]{} circle[radius=2pt];
        \draw (3.2,0.8) circle[radius=2pt]; and \fill (3.2,0.8) node[above = 2pt]{} circle[radius=2pt];
        \draw (4,0.8) circle[radius=2pt]; and \fill (4,0.8) node[above = 2pt]{} circle[radius=2pt];
        \draw (4.8,0.8) circle[radius=2pt]; and \fill (4.8,0.8) node[above = 2pt]{} circle[radius=2pt];
        \draw (5.6,0.8) circle[radius=2pt]; and \fill (5.6,0.8) node[above = 2pt]{} circle[radius=2pt];
        \draw (6.4,0.8) circle[radius=2pt]; and \fill (6.4,0.8) node[above = 2pt]{} circle[radius=2pt];
        \draw (0.8,1.6) circle[radius=2pt]; and \fill (0.8,1.6) node[above = 2pt]{} circle[radius=2pt];
        \draw (1.6,1.6) circle[radius=2pt]; and \fill (1.6,1.6) node[above = 2pt]{} circle[radius=2pt];
        \draw (2.4,1.6) circle[radius=2pt]; and \fill (2.4,1.6) node[above = 2pt]{} circle[radius=2pt];
        \draw (3.2,1.6) circle[radius=2pt]; and \fill (3.2,1.6) node[above = 2pt]{} circle[radius=2pt];
        \draw (4,1.6) circle[radius=2pt]; and \fill (4,1.6) node[above = 2pt]{} circle[radius=2pt];
        \draw (4.8,1.6) circle[radius=2pt]; and \fill (4.8,1.6) node[above = 2pt]{} circle[radius=2pt];
        \draw (5.6,1.6) circle[radius=2pt]; and \fill (5.6,1.6) node[above = 2pt]{} circle[radius=2pt];
        \draw (0.8,2.4) circle[radius=2pt]; and \fill (0.8,2.4) node[above = 2pt]{} circle[radius=2pt];
        \draw (1.6,2.4) circle[radius=2pt]; and \fill (1.6,2.4) node[above = 2pt]{} circle[radius=2pt];
        \draw (2.4,2.4) circle[radius=2pt]; and \fill (2.4,2.4) node[above = 2pt]{} circle[radius=2pt];
        \draw (3.2,2.4) circle[radius=2pt]; and \fill (3.2,2.4) node[above = 2pt]{} circle[radius=2pt];
        \draw (4,2.4) circle[radius=2pt]; and \fill (4,2.4) node[above = 2pt]{} circle[radius=2pt];
        \draw (4.8,2.4) circle[radius=2pt]; and \fill (4.8,2.4) node[above = 2pt]{} circle[radius=2pt];
        \draw (0.8,3.2) circle[radius=2pt]; and \fill (0.8,3.2) node[above = 2pt]{} circle[radius=2pt];
        \draw (1.6,3.2) circle[radius=2pt]; and \fill (1.6,3.2) node[above = 2pt]{} circle[radius=2pt];
        \draw (2.4,3.2) circle[radius=2pt]; and \fill (2.4,3.2) node[above = 2pt]{} circle[radius=2pt];
        \draw (3.2,3.2) circle[radius=2pt]; and \fill (3.2,3.2) node[above = 2pt]{} circle[radius=2pt];
        \draw (4,3.2) circle[radius=2pt]; and \fill (4,3.2) node[above = 2pt]{} circle[radius=2pt];
        \draw (0.8,4) circle[radius=2pt]; and \fill (0.8,4) node[above = 2pt]{} circle[radius=2pt];
        \draw (1.6,4) circle[radius=2pt]; and \fill (1.6,4) node[above = 2pt]{} circle[radius=2pt];
        \draw (2.4,4) circle[radius=2pt]; and \fill (2.4,4) node[above = 2pt]{} circle[radius=2pt];

        \draw (6.7,0) circle[radius=0pt]; and \fill (6.7,0) node[below = 1pt]{$\frac{l_1R}{\pi}$} circle[radius=0pt];
        \draw (0,4.4) circle[radius=0pt]; and \fill (0,4.4) node[left = 1pt]{$\frac{l_2R}{\pi}$} circle[radius=0pt];

        \draw[-] (0,0.8) -- (6.4,0.8);
        \draw[-] (0,1.6) -- (5.6,1.6);
        \draw[-] (0,2.4) -- (4.8,2.4);
        \draw[-] (0,3.2) -- (4,3.2);
        \draw[-] (0,4) -- (2.4,4);
        \draw[-] (0.8,0) -- (0.8,4);
        \draw[-] (1.6,0) -- (1.6,4);
        \draw[-] (2.4,0) -- (2.4,4);
        \draw[-] (3.2,0) -- (3.2,3.2);
        \draw[-] (4,0) -- (4,3.2);
        \draw[-] (4.8,0) -- (4.8,2.4);
        \draw[-] (5.6,0) -- (5.6,1.6);
        \draw[-] (6.4,0) -- (6.4,0.8);
        
    \end{tikzpicture}
    \caption{A quarter ellipse whose radii are proportional to $R$. Each lattice point in its interior corresponds to a unit square. The total area of the squares and the area of the quarter ellipse both grow on the order $R^2$ as $R \to \infty$, while the error is of the order $R$. Hence the number of lattice points in the interior of the quarter ellipse is asymptotically equal to its area as $R \to \infty$.}
    \label{fig:lattice_points}
\end{figure}

    To see this, note that each lattice point $(m,n)$ corresponds to a square of length 1 as shown in Figure \ref{fig:lattice_points}. Since two such squares have disjoint interiors, and each square has volume 1, $N(R^2)$ equals the number of squares which fit inside the quarter ellipse.  This deviates from its area by less than the length of its arc, which is of the order $R$ as $R \to \infty$. Similarly, number of the lattice points on the boundary of the ellipse is of the order $R$ as $R \to \infty$. It therefore  follows that the error terms which are of order $R$ can be gleefully discarded, and we have 
    \begin{equation*}
        \lim_{R \to \infty} \frac{N(R^2)}{|E_R|} = 1,
    \end{equation*} 
    
Thus, we get
    \begin{equation*}
        N(R^2) \sim \frac{l_1l_2R^2}{4\pi} \iff 
        N(\lambda) \sim \frac{l_1l_2\lambda}{4\pi}, \,\,\lambda \to \infty.
    \end{equation*}
    Noting that $|I| = l_1l_2$ completes the proof. 
\end{proof}

What we have done so far with Weyl's law can easily be generalized to $n$-dimensional Euclidean space. Indeed, if $\Omega = \prod_{i=1}^n (0,l_i) \subset \R^n$ is an $n$-dimensional rectangle, then for Dirichlet boundary conditions the eigenvalues become
\begin{equation*}
    \sum_{i = 1}^n \frac{m_i^2\pi^2}{l_i^2}
\end{equation*}
with corresponding eigenfunctions
\begin{equation*}
    \prod_{i=1}^n \sin(\frac{m_i \pi x_i}{l_i}), \,\, m_i \geq 1, \,\, x_i \in (0, l_i).
\end{equation*}
Similarly, for Neumann boundary conditions we get the same eigenvalues except that we now have $m_i \geq 0$ instead. The eigenfunctions become
\begin{equation*}
    \prod_{i=1}^n \cos(\frac{m_i \pi x_i}{l_i}), \,\, m_i \geq 0, \,\, x_i \in (0, l_i).
\end{equation*}
This is proven by separating variables in the $n$-dimensional eigenvalue problem analogous to the two-dimensional case. From this, we can analogously deduce Weyl's law from the proof of Proposition \ref{prop:Weyl_rectangles}.  We see that $N(R^2)$ becomes asymptotically equal to the volume of 
\begin{equation*}
    \left\{(x_1,\dots,x_n) : \sum_{i=1}^n \frac{x_i^2\pi^2}{l_i^2} \leq R^2, \,\, x_i \geq 0\right\}.
\end{equation*}
This is, by a standard calculation,
\begin{equation*}
    \frac{\omega_nR^n}{(2\pi)^n} \prod_{i=1}^n l_i,
\end{equation*}
where $\omega_n$ denotes the volume of the $n$-dimensional unit ball.
Thus,
\begin{equation*}
    N(\lambda) \sim \frac{\omega_n|I|}{(2\pi)^n}\lambda^{n/2}, \,\, \lambda \to \infty,
\end{equation*}
which shows that
\begin{equation*}
    \lim_{\lambda \to \infty} \frac{N(\lambda)}{\lambda^{n/2}} = \frac{\omega_n|I|}{(2\pi)^n}.
\end{equation*}

Our next goal is to show that Weyl's law holds for arbitrary bounded, connected domains in $\R^n$. To do this, we need the variational principles, also known as Rayleigh-Ritz principles, which were developed around the same time as Lorentz's lectures in 1910 \cite{leissa}, \cite{Rayleigh2011}, \cite{Ritz1909}.

\begin{theorem} \label{Theorem:Variational_principles} 
    Let $\Omega$ be a domain with piecewise smooth boundary. If $\lambda_k$ denotes the $k$th eigenvalue of the Laplace operator with corresponding eigenfunction $f_k$, then
    \begin{align}
        \lambda_k &= 
        \begin{cases}
            \inf_{\substack{f \in H \\ f \neq 0}} \frac{\int_\Omega |\nabla f|^2}{\int_\Omega |f|^2} \text{ and is attained by } f_1, \text{ if } k = 1, \\
            \inf_{\substack{f \in H \\ f \neq 0}}  \left\{\frac{\int_\Omega |\nabla f|^2}{\int_\Omega |f|^2} : \langle f, f_j \rangle_{L^2(\Omega)} = 0, \,\,j = 1,\dots, k-1 \right\} \\\text{and is attained by } f_k, \text{ if } k \geq 2,
        \end{cases} \label{eq:variational_principle1}
        \\
        \lambda_k &= \sup_{\substack{L \subset H \\ \text{ dim } L = k-1}} \inf_{\substack{f \in L^{\perp} \\ f \neq 0}} \frac{\int_\Omega |\nabla f|^2}{\int_\Omega |f|^2}, \label{eq:variational_principle2}\\
        \lambda_k &= \inf_{\substack{L \subset H \\ \text{ dim } L = k}} \sup_{\substack{f \in L \\ f \neq 0}} \frac{\int_\Omega |\nabla f|^2}{\int_\Omega |f|^2}, \label{eq:variational_principle3}
    \end{align}
    where the integrals are taken with the respect to the Lebesgue measure. Here, $H = H_0^1(\Omega)$ for Dirichlet boundary conditions, and $H = H^1(\Omega)$ for Neumann boundary conditions. 
\end{theorem}
\begin{remark}
    The expression 
    \begin{equation*}
        \frac{\int_\Omega |\nabla f|^2}{\int_\Omega |f|^2}
    \end{equation*}
    may be called the \textit{Rayleigh quotient,} \textit{Ritz quotient,} or \textit{Rayleigh-Ritz quotient} of $f$. 
\end{remark}

As mentioned earlier, Weyl developed his law as part of his broader research in mathematical physics and differential geometry, particularly in the context of elliptic partial differential equations and the spectrum of differential operators. While it is difficult to pinpoint exactly how he discovered or learned about the variational principles, we expect that he was familiar with Rayleigh's works on the principles of sound \cite{Rayleigh2011} and may also have accessed Ritz's methods published in 1909 \cite{Ritz1909}.

We can now show that Weyl's law holds for an arbitrary domain $\Omega$ by approximating it with rectangles and using the Dirichlet and Neumann eigenvalues on the approximating rectangles together with the variational principles, following Weyl's original proof \cite{weyl1912asymptotische}. This is known as Dirichlet-Neumann bracketing. We will show the proof in detail for planar domains, and afterwards we will explain briefly how to obtain the analogous result for domains in $\R^n$.

\begin{theorem}[Weyl's law] \label{theorem:Weyl}
    Let $\Omega \subset \R^2$ be a bounded domain and let $0 < \lambda_1 \leq \lambda_2 \leq \dots$ be the Dirichlet eigenvalues of the Laplace operator on $\Omega$. If the boundary $\partial\Omega$ is piecewise smooth, then
    \begin{equation*}
        \lim_{\lambda \to \infty} \frac{N(\lambda)}{\lambda} = \frac{|\Omega|}{4\pi}.
    \end{equation*}
\end{theorem}
\begin{proof}
    \begin{figure}
    \centering
    \begin{tikzpicture} 
        \draw plot [smooth cycle] coordinates {(4,2)  (3,5) (5,7) (8,7.7)  (9,6)  (7,4)   (8,2)};
        \draw[-] (4,1.5) -- (8,1.5); 
        \draw[-] (3.5,2) -- (8.5,2);
        \draw[-] (3,2.5) -- (8.5,2.5);
        \draw[-] (3,3) -- (8,3);
        \draw[-] (3,3.5) -- (7.5,3.5);
        \draw[-] (3,4) -- (7.5,4);
        \draw[-] (2.5,4.5) -- (8,4.5);
        \draw[-] (2.5,5) -- (9,5);
        \draw[-] (3,5.5) -- (9,5.5);
        \draw[-] (3,6) -- (9.5,6);
        \draw[-] (3.5,6.5) -- (9.5,6.5);
        \draw[-] (4,7) -- (9,7);
        \draw[-] (5,7.5) -- (9,7.5);
        \draw[-] (6,8) -- (8.5,8);

        \draw[-] (2.5,4.5) -- (2.5,5); 
        \draw[-] (3,2.5) -- (3,6);
        \draw[-] (3.5,2) -- (3.5,6.5);
        \draw[-] (4,1.5) -- (4,7);
        \draw[-] (4.5,1.5) -- (4.5,7);
        \draw[-] (5,1.5) -- (5,7.5);
        \draw[-] (5.5,1.5) -- (5.5,7.5);
        \draw[-] (6,1.5) -- (6,8);
        \draw[-] (6.5,1.5) -- (6.5,8);
        \draw[-] (7,1.5) -- (7,8);
        \draw[-] (7.5,1.5) -- (7.5,8);
        \draw[-] (8,1.5) -- (8,3); \draw[-] (8,4.5) -- (8,8);
        \draw[-] (8.5,2) -- (8.5,2.5); \draw[-] (8.5,5) -- (8.5,8);
        \draw[-] (9,5) -- (9,7.5);
        \draw[-] (9.5,6) -- (9.5,6.5);

        \draw[-] (3,4.5) -- (5.5,7); 
        \draw[-] (3.5,4.5) -- (6,7);
        \draw[-] (3.5,4) -- (7,7.5);
        \draw[-] (3.5,3.5) -- (7.5,7.5);
        \draw[-] (3.5,3) -- (8,7.5);
        \draw[-] (4,3) -- (8,7);
        \draw[-] (4,2.5) -- (8.5,7);
        \draw[-] (4,2) -- (8.5,6.5);
        \draw[-] (4.5,2) -- (7,4.5); \draw[-] (7.5,5) -- (8.5,6);
        \draw[-] (5,2) -- (7,4);
        \draw[-] (5.5,2) -- (7,3.5);
        \draw[-] (6,2) -- (7,3);
        \draw[-] (6.5,2) -- (7.5,3);
        \draw[-] (7,2) -- (7.5,2.5);

        \draw[-] (4,2.5) -- (4.5,2); 
        \draw[-] (3.5,3.5) -- (5,2);
        \draw[-] (3.5,4) -- (5.5,2);
        \draw[-] (3,5) -- (6,2);
        \draw[-] (3.5,5) -- (6.5,2);
        \draw[-] (3.5,5.5) -- (7,2);
        \draw[-] (4,5.5) -- (7.5,2);
        \draw[-] (4,6) -- (7.5,2.5);
        \draw[-] (4.5,6) -- (7,3.5);
        \draw[-] (4.5,6.5) -- (7,4);
        \draw[-] (5,6.5) -- (7,4.5);
        \draw[-] (5,7) -- (7,5);
        \draw[-] (5.5,7) -- (7.5,5);
        \draw[-] (6,7) -- (8,5);
        \draw[-] (6.5,7) -- (8,5.5);
        \draw[-] (6.5,7.5) -- (8.5,5.5);
        \draw[-] (7,7.5) -- (8.5,6);
        \draw[-] (7.5,7.5) -- (8.5,6.5);
    \end{tikzpicture}
    \caption{The domain $\Omega$ has an inner and outer covering consisting of finitely many rectangles. The rectangles contained in the interior are marked with a cross.}
    \label{fig:inner_outer_covering_rectangles}
\end{figure}
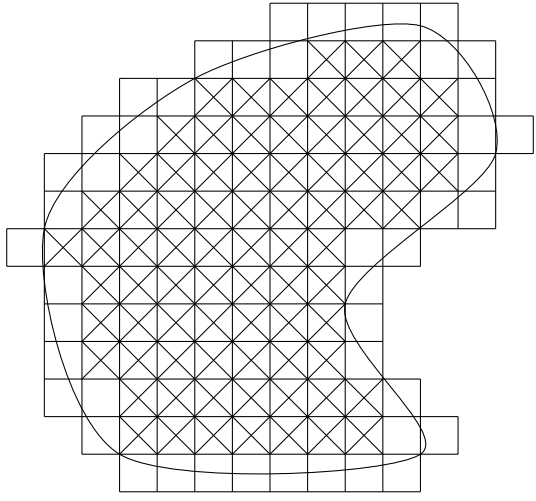

    Let $\epsilon > 0$.  We choose an inner covering $\Omega_I = \bigcup_{i=1}^N \Omega_{I,i}$ and an outer covering $\Omega_O = \bigcup_{j=1}^M \Omega_{O,j}$ each consisting of finitely many disjoint rectangles; they may only overlap on their boundaries. We choose these coverings such that the area of the region between the outer and inner coverings has area less than $\epsilon$, that is $|\Omega_0 \backslash \Omega_1| < \epsilon$.  This is shown in Figure \ref{fig:inner_outer_covering_rectangles}.
        
    Let $\lambda_k(\Omega_\cdot)$ and $\mu_k(\Omega_\cdot)$ denote the $k$th Dirichlet and Neumann eigenvalues, respectively, on $\Omega_\cdot$. We take the unions 
    \begin{equation*}
        \bigcup_{i,k} \{\lambda_k(\Omega_{I,i})\}, \,\,\bigcup_{j,k} \{\lambda_k(\Omega_{O,j})\}, \,\,\bigcup_{i,k} \{\mu_k(\Omega_{I,i})\}, \,\, \bigcup_{j,k} \{\lambda_k(\Omega_{O,j})\} 
    \end{equation*}
    of the respective eigenvalues over the rectangles and reorder them into increasing sequences. Denote the results by $\{\Tilde{\lambda}_k(\Omega_I)\}$, $\{\Tilde{\lambda}_k(\Omega_O)\}$, $\{\Tilde{\mu}_k(\Omega_I)\}$, and $\{\Tilde{\mu}_k(\Omega_O)\}$, respectively.  We now require the following lemma.
    \begin{lemma} \label{lemma:weyl_inequalities}
        With the notations above, we have for every $k$
        \begin{equation*}
            \Tilde{\mu}_k(\Omega_O) \stackrel{(1)}{\leq} \mu_k(\Omega_O) \stackrel{(2)}{\leq} \lambda_k(\Omega_O) \stackrel{(3)}{\leq} \lambda_k(\Omega) \stackrel{(4)}{\leq} \lambda_k(\Omega_I) \stackrel{(5)}{\leq} \Tilde{\lambda}_k(\Omega_I).
        \end{equation*}
    \end{lemma}
    \begin{proof}[Proof of lemma \ref{lemma:weyl_inequalities}] There are five inequalities to prove, (1) through (5). We start with (3). Since $\Omega \subset \Omega_O$, we can extend functions $f \in H_0^1(\Omega)$ to equal 0 on $\Omega_O \backslash \Omega$. Then we have $f \in H_0^1(\Omega_O)$, and so $H_0^1(\Omega) \subset H_0^1(\Omega_O)$. Now (3) follows immediately from the variational principle \eqref{eq:variational_principle3} because taking the infimum over a larger set yields a smaller result. Inequality (4) follows in the same way since $\Omega_I \subset \Omega$. In fact, the same argument works for (5) because for each rectangle we can extend its $H_0^1$ functions to be zero outside it, and then it will be included in $\Omega_I$. Similarly, (2) follows from the fact that for Neumann boundary conditions \eqref{eq:variational_principle3} minimizes over $H^1$, which is larger than what we minimize over for Dirichlet boundary conditions, namely $H_0^1$. 

    Thus, it remains to show (1). Recall that for Neumann boundary conditions on a rectangle, 0 is an eigenvalue with corresponding eigenfunctions being non-zero constant functions. Thus, if we let $R_1,\dots,R_M$ denote the rectangles of $\Omega_O$, then each $R_i$ has 0 as an eigenvalue. This means that $\Tilde{\mu}_1(\Omega_O) = \dots = \Tilde{\mu}_M(\Omega_O) = 0$. Since $\mu_i(\Omega_O) \geq 0$ for every $i$, it follows that $\Tilde{\mu}_i(\Omega_O) \leq \mu_i(\Omega_O)$ for $i = 1,\dots, M$. The next eigenvalue, $\Tilde{\mu}_{M+1}(\Omega_O)$, will then be the smallest positive eigenvalue coming from all the rectangles $R_1,\dots,R_M$. After re-labeling, we may assume that $\Tilde{\mu}_{M+1}(\Omega_O)$ comes from rectangle $R_1$. 

    The next step is to use the variational principle \eqref{eq:variational_principle2}. We have
    \begin{equation} \label{eq:variational2_mu_M+1}
        \mu_{M+1}(\Omega_O) = \sup_{\substack{L \subset H^1(\Omega_O) \\ \text{ dim } L = M}} \inf_{\substack{f \in L^{\perp} \\ f \neq 0}} \frac{\int_{\Omega_O} |\nabla f|^2}{\int_{\Omega_O} |f|^2}.
    \end{equation}
    Now let $\Tilde{L}$ be the $M$-dimensional subspace of $H^1(\Omega_O)$ spanned by the characteristic functions on the rectangles, so that a basis of $\Tilde{L}$ is given by $\{f_1,\dots,f_M\}$ where 
    \begin{equation*}
        f_i(x) = 
        \begin{cases}
            1, \,\,x \in R_i, \\
            0, \,\,x \notin R_i. 
        \end{cases}
    \end{equation*}
    To show that $\Tilde{\mu}_{M+1}(\Omega_O) \leq \mu_{M+1}(\Omega_O)$, by \eqref{eq:variational2_mu_M+1} it's enough to show that
    \begin{equation*}
        \Tilde{\mu}_{M+1}(\Omega_O) \leq \inf_{\substack{f \in \Tilde{L}^{\perp} \\ f \neq 0}} \frac{\int_{\Omega_O} |\nabla f|^2}{\int_{\Omega_O} |f|^2}. 
    \end{equation*}
    Note that the eigenfunction $u$ corresponding to $\Tilde{\mu}_{M+1}(\Omega_O)$ is in $\Tilde{L}^{\perp}$. Indeed, this is immediate for $R_2,\dots,R_M$ since we extend $u$ to be 0 on those. For $R_1$, recall the variational principle \eqref{eq:variational_principle1} which says that $u$ is orthogonal to the eigenfunctions corresponding to $\Tilde{\mu}_1,\dots,\Tilde{\mu}_M$. These eigenfunctions are constant functions, so we obtain $\int_{R_1} u = 0$ and therefore $u$ is also orthogonal to the characteristic functions on $R_1$, hence $u \in \Tilde{L}^\perp$. 

    In fact, for any $f \in \Tilde{L}^\perp$ we have $\int_{R_i} f = 0$ for every $i$ by definition of $\Tilde{L}$. In particular, $f$ is orthogonal to the eigenfunctions corresponding to $\Tilde{\mu}_1(\Omega_O), \dots, \Tilde{\mu}_M(\Omega_O)$. This gives, by \eqref{eq:variational_principle1}, that for every $i$
    \begin{equation*}
        \Tilde{\mu}^{(i)} \leq \frac{\int_{R_i} |\nabla f|^2}{\int_{R_i} |f|^2},
    \end{equation*}
    where $\Tilde{\mu}^{(i)}$ is the first positive Neumann eigenvalue on $R_i$. Then
    \begin{equation*}
        \int_{R_i} |\nabla f|^2 \geq \Tilde{\mu}^{(i)} \int_{R_i} |f|^2 \geq \Tilde{\mu}_{M+1}(\Omega_O) \int_{R_i} |f|^2,
    \end{equation*}
    and so
    \begin{equation*}
        \frac{\int_{\Omega_O} |\nabla f|^2}{\int_{\Omega_O} |f|^2} = \frac{\sum_{i=1}^M \int_{R_i} |\nabla f|^2}{\sum_{i=1}^M \int_{R_i} |f|^2} \geq \frac{\Tilde{\mu}_{M+1}(\Omega_O) \sum_{i=1}^M \int_{R_i} |f|^2}{\sum_{i=1}^M \int_{R_i} |f|^2} = \Tilde{\mu}_{M+1}(\Omega_O).
    \end{equation*}
    Since this holds for all $f \in \Tilde{L}^\perp$, it also holds for the infimum:
    \begin{equation*}
        \Tilde{\mu}_{M+1}(\Omega_O) \leq \inf_{\substack{f \in \Tilde{L}^\perp \\ f \neq 0}} \frac{\int_{\Omega_O} |\nabla f|^2}{\int_{\Omega_O} |f|^2}.
    \end{equation*}
    This completes the proof that $\Tilde{\mu}_{M+1}(\Omega_O) \leq \mu_{M+1}(\Omega_O)$. Inequality (1) now follows by induction on $k$.
    \end{proof}
    Continuing with the proof of the theorem, let $N_N(\lambda, \Omega_\cdot)$ and $N_D(\lambda, \Omega_\cdot)$ denote the number of Neumann and Dirichlet eigenvalues less than or equal to $\lambda$ on $\Omega_\cdot$, respectively. By Lemma \ref{lemma:weyl_inequalities}, we have
    \beq &    \frac{\sum_{i=1}^N N_D(\lambda, \Omega_{I,i})}{\lambda} \leq \frac{N_D(\lambda, \Omega_I)}{\lambda} \leq \frac{N_D(\lambda, \Omega)}{\lambda} \nn \\ 
    &\leq \frac{N_D(\lambda, \Omega_O)}{\lambda} \leq \frac{N_N(\lambda, \Omega_O)}{\lambda} \leq \frac{\sum_{j=1}^M N_N(\lambda, \Omega_{O,j})}{\lambda}
    \nn \eeq 
    because having larger eigenvalues means that there are fewer eigenvalues up to some fixed number. By Weyl's law on rectangles (see Proposition \ref{prop:Weyl_rectangles}), 
    \begin{equation*}
         \lim_{\lambda \to \infty} \frac{\sum_{i=1}^N N_D(\lambda, \Omega_{I,i})}{\lambda} = \frac{\sum_{i = 1}^N|\Omega_{I,i}|}{4\pi} = \frac{|\Omega_I|}{4\pi},
    \end{equation*}
    and similarly
    \begin{equation*}
         \lim_{\lambda \to \infty} \frac{\sum_{j=1}^M N_N(\lambda, \Omega_{O,j})}{\lambda} = \frac{|\Omega_O|}{4\pi}.
    \end{equation*}
    Thus, we obtain
    \begin{equation*}
        \frac{|\Omega_I|}{4\pi} \leq \liminf_{\lambda \to \infty} \frac{N_D(\lambda, \Omega)}{\lambda} \leq \limsup_{\lambda \to \infty} \frac{N_D(\lambda, \Omega)}{\lambda} \leq \frac{|\Omega_O|}{4\pi}.
    \end{equation*}
    Letting $\epsilon \to 0$ gives
    \begin{equation*}
        \lim_{\lambda \to \infty} \frac{N_D(\lambda, \Omega)}{\lambda} = \frac{|\Omega|}{4\pi},
    \end{equation*}
    as desired. 
    \end{proof}

If we had an $n$-dimensional domain $\Omega \subset \R^n$, we could analogously approximate it by an inner and outer covering of $n$-dimensional rectangles. Since both Weyl's law and the variational principles hold for such rectangles, the proof would look exactly the same, and we would obtain
\begin{equation*}
    \lim_{\lambda \to \infty} \frac{N(\lambda)}{\lambda^{n/2}} = \frac{\omega_n|\Omega|}{(2\pi)^n}.
\end{equation*}

Weyl's law formulated as in Theorem \ref{theorem:Weyl} also holds for Neumann boundary conditions. However, the proof is more difficult than the Dirichlet case, partly because there is no obvious analogue of Dirichlet-Neumann bracketing. In particular, two domains $\Omega_1$ and $\Omega_2$ satisfying $\Omega_1 \subset \Omega_2$ need not satisfy $\lambda_k(\Omega_2) \leq \lambda_k(\Omega_1)$ in the Neumann case, as they do in the Dirichlet case. In light of Theorem \ref{Theorem:Variational_principles}, while functions $f \in H_0^1(\Omega_1)$ can easily be extended to being in $H_0^1(\Omega_2)$ by defining them to be zero in $\Omega_2 \backslash \Omega_1$, functions $f \in H^1(\Omega_1)$ cannot always be extended to $H^1(\Omega_2)$.

Interestingly, while we only considered domains with piecewise smooth boundary, Weyl's law can be proven to hold for any bounded domain in $\R^n$ with Dirichlet boundary conditions. This is, however, not true in the Neumann case, where the domain in addition needs to satisfy the so called \textit{extension property}. See \cite[Ch. 3]{frank2022schrodinger} for details. See also \cite{jakvsic1992eigenvalue} for examples of domains which do not satisfy Weyl's law in the Neumann case.

As mentioned in \S\ref{sec:introduction}, Weyl's law shows that the spectrum of an $n$-dimensional domain determines its dimension $n$ and its $n$-dimensional volume. A natural question then is whether the spectrum determines \textit{all} of the geometry of a domain, or more generally a Riemannian manifold. As one might guess, this is a hard problem because for for most Riemannian manifolds one cannot compute the spectrum explicitly. There is, however, one particular class of examples for which one can explicitly compute the spectrum in a certain sense, namely \textit{flat tori}. 

\section{Listening to flat tori} \label{sec:quadratic_forms}
An $n$-dimensional flat torus is a smooth compact Riemannian manifold obtained by taking the quotient of Euclidean space $\R^n$ by a full-rank lattice.  For inclusivity, we recall that  $\GL_n(\R)$ is the group of all invertible $n\times n$ matrices with entries in $\R$ with group operation given by matrix multiplication. The subgroup $\GL_n(\Z) \subset \GL_n(\R)$ consists of those matrices with entries in $\Z$ whose inverses also have entries in $\Z$.  A full-rank lattice in $\R^n$ is a discrete additive subgroup of the additive group $\R^n$, that necessarily is a set of the form $\Gamma = A\Z^n$ for some invertible matrix $A \in \R^{n \times n}$ called a \textit{basis matrix} of $\Gamma$.  We note that there are many basis matrices for the same lattice because for any $B \in \GL_n(\Z)$, the sets $A \Z^n$ and $AB \Z^n$ are identical. The lattice $\Gamma$ defines an associated \textit{flat torus} $\T_\Gamma = \R^n / \Gamma$. Such a torus is ``flat'' because its Riemannian metric is induced by the Euclidean metric on $\R^n$. Two flat tori are \textit{isometric} if they are isometric as Riemannian manifolds.

Thanks to the structure of a flat torus as a quotient of Euclidean space, the Laplace eigenvalue problem for a flat torus can be studied in Euclidean space by searching for all twice weakly locally differentiable functions  $f : \R^n \to \C$ and $\lambda \in \C$ such that for all $x \in \R^n$ and for all $\gamma \in \Gamma$,
\begin{equation} \label{eq:Eigenvalue_flat_torus}
        \begin{cases}
            \Delta f(x) = \lambda f(x), \\
            f(x + \gamma) = f(x), \\ 
            \nabla f(x + \gamma) = \nabla f(x). \\
            
        \end{cases}
\end{equation}

The idea of two lattices being the same shape is captured by the notion of congruence. Let $O_n(\R)$ be the group of all invertible $n\times n$ matrices $A$ with real entries such that $A^{-1} = A^T$. We say that two lattices $A_1\Z^n$ and $A_2\Z^n$ are \textit{congruent} if $A_1 = A_2B$ for some $B \in O_n(\R)$. By \cite[p. 5]{berger1971spectre}, two flat tori are isometric if and only if their associated lattices are congruent. 

To find the spectrum of a flat torus, we need the notion of dual lattices. The \textit{dual lattice} $\Gamma^*$ of a lattice $\Gamma$ is 
\begin{equation*}
    \Gamma^* = \{\gamma \in \Gamma : \gamma \cdot \delta \in \Z \text{ for all } \delta \in \Gamma\}.
\end{equation*}
\begin{theorem}[\cite{nilsson2022isospectral}, Thm. 2.8]
    The spectrum of a flat torus $\R^n / \Gamma$ is the multi-set (meaning counting multiplicity) 
    \begin{equation*}
        \{4\pi^2||\gamma||^2 : \gamma \in \Gamma^*\}.
    \end{equation*}
    The multiplicity of $4\pi^2||\gamma||^2$ is the number of different $\delta \in \Gamma^*$ such that $||\delta|| = ||\gamma||$.
\end{theorem}

In turns out that quadratic forms have a strong connection to flat tori. To understand this, we need some more definitions. Recall that any quadratic form $q$ in $n$ variables has an \textit{associated matrix}, that is, a symmetric $n \times n$ matrix $Q$ such that $q(x) = x^TQx$ for all $x \in \R^n$. Moreover, $q$ is \textit{positive definite} if $q(x) \geq 0$ for all $x \in \R^n$ with equality if and only if $x = 0$. 
It is \textit{rational} if all entries in $Q$ are rational, and it is \textit{even} if all entries in $Q$ are integers, and the diagonal elements are even.  

\begin{proposition}[\cite{serre2000matrices}, Thm. 11.2]
    If $q(x) = x^TQx$ is a positive definite quadratic form in $n$ variables, then there is an invertible $n \times n$ matrix $A$ such that $Q = A^TA$. This is called the \textup{Cholesky factorization} of $Q$. The matrix $A$ is unique up to multiplication by matrices in $O_n(\R)$.
\end{proposition}
If $q$ is a positive definite quadratic form in $n$ variables whose associated matrix $Q$ has Cholesky factorization $Q = A^TA$, then $A\Z^n$ is called the \textit{underlying lattice} of $Q$ and $q$. Since $A\Z^n$ may not equal $AO\Z^n$ for $O \in O_n(\R)$, the underlying lattice of a quadratic form is not unique. However, they are all congruent, and therefore their associated flat tori are all isometric.

Two quadratic forms $q(x) = x^TQx$ and $p(x) = x^TPx$ are called \textit{integrally equivalent} if $B^TQB = P$ for some $B \in \GL_n(\Z)$. This defines an equivalence relation, and two integrally equivalent quadratic forms have identical \textit{representation numbers}. That is, $R(q,t) = R(p,t)$ for all $t \geq 0$, where 
\begin{equation*}
    R(q,t) = \#\{x \in \Z^n : q(x) = t\}.
\end{equation*}
Conversely, while the two quadratic forms with associated matrices $A^TA$ and $(AB)^T$ $AB$ may not be equal for $B \in \GL_n(\Z)$, they are integrally equivalent and therefore have identical representation numbers. This makes it reasonable to associate a flat torus $\R^n / \Gamma$, $\Gamma = A\Z^n$, with the equivalence class of integrally equivalent quadratic forms which contains the quadratic form with associated matrix $A^TA$. It follows, then, that two flat tori are isometric if and only if their associated equivalence classes of quadratic forms coincide. Moreover, we have the following result.

\begin{theorem}[\cite{nilsson2022isospectral}, Cor. 2.15]
    Two flat tori are isospectral if and only if their representation numbers associated as above are identical.
\end{theorem}
We say that two quadratic forms are \textit{isospectral} if they have identical representation numbers. Thus, two quadratic forms are isospectral if and only if their associated flat tori are isospectral.  If two flat tori are isospectral, are they necessarily isometric?  In other words, if we are listening to flat tori, can we hear their shapes?

This seems like a difficult question because verifying isospectrality directly from the spectrum requires checking that all of the infinitely many eigenvalues are identical.  Checking isometry - or non-isometry - is also not so simple since there are infinitely many isometries of Euclidean space.  One simple way to check non-isometry is to show that one lattice is reducible whereas the other lattice is not.  A lattice $\Gamma \subset \R^n$ is a \textit{direct sum} of two lattices $\Gamma_1, \Gamma_2$ if 
\begin{align*}
    \Gamma_1 + \Gamma_2 &:= \{x + y : x \in \Gamma_1, y \in \Gamma_2\} = \Gamma, \\
    \Gamma_1 \cdot \Gamma_2 &:= \{x \cdot y : x \in \Gamma_1, y \in \Gamma_2\} = \{0\}.
\end{align*}
We write this as $\Gamma = \Gamma_1 \oplus \Gamma_2$. The lattice $\Gamma$ is called \textit{reducible} if it is the direct sum of two lattices of lower dimension. Otherwise it is \textit{irreducible}. It is the case that $\Gamma \neq \{0\}$ is reducible if and only if it is congruent to $\Gamma_1 \times \Gamma_2$, where both $\Gamma_1,\Gamma_2$ have dimension at least 1 \cite[Lemma 2.26]
{nilsson2022isospectral}. 

\begin{proposition}[\cite{nilsson2022isospectral}, Prop. 2.28] \label{prop:reducible_lattice_products}
    Suppose a lattice $\Gamma \in \R^n$ can be decomposed as
    \begin{equation*}
        \Gamma = \Gamma_1 \oplus \dots \oplus \Gamma_k.
    \end{equation*}
    Then, a lattice $\Lambda$ is congruent to $\Gamma$ if and only if $\Lambda$ is a direct sum of $k$ lattices $\Lambda_i$ congruent to $\Gamma_i$. In particular, if $\Gamma$ is reducible and $\Lambda$ is irreducible, then they are not congruent.
\end{proposition}
It follows that if $\Gamma_1,\dots,\Gamma_k$ and $\Lambda_1,\dots,\Lambda_{k'}$ are all irreducible, then the products $\Gamma_1 \times \dots \times \Gamma_k$ and $\Lambda_1 \times \dots \times \Lambda_{k'}$ are congruent if and only if $k = k'$ and $\Gamma_i,\Lambda_i$ are congruent after possibly reordering.  To show that two lattices (or quadratic forms) are isospectral, we use the following result.
\begin{theorem}[\cite{nilsson2022isospectral}, Cor. 3.7] \label{theorem:Isospectrality_certificate}
    Let $P$ and $Q$ be the associated matrices of two even positive definite quadratic forms in $2k$ variables. Define for $N \geq 1$
    \begin{equation*}
        \mu_0(N) := N\prod_{\substack{p \text{ prime } \\ p | N}}\biggl(1 + \frac{1}{p}\biggr),
    \end{equation*}
    and let $N_P$ and $N_Q$ be the smallest positive integers such that $N_PP^{-1}$ and $N_QQ^{-1}$ are even, respectively. Then the quadratic forms are isospectral if and only if $det(P) = det(Q)$, $N_P = N_Q$, and their representation numbers $R(p,t)$ and $R(q,t)$ for $t \in [0, \frac{\mu_0(N_P)k}{12} + 1]$ are identical. 
\end{theorem}
We now describe Milnor's example of two isospectral non-isometric Riemannian manifolds. Let $e_1,\dots,e_n$ denote the standard basis vectors of $\R^n$, write $\mathds{1} = \sum_{i=1}^n e_i$, and define the lattice
\begin{equation*}
    E_n = \biggl\{x \in \Z^n \cup \biggl(\frac{1}{2}\mathds{1} + \Z\biggr) : \sum_{i=1}^n x_i \in 2\Z\biggr\}.
\end{equation*}
A basis matrix for $E_{4n}$ is given by the $4n \times 4n$ matrix
\begin{equation*}
   (A_{4n})_{ij} =
   \begin{cases}
       2, \text{ if } i = j = 1, \\
       1, \text{ if } i = j, \,\,2 \leq i \leq 4n-1, \\
       -1, \text{ if } j = i + 1, \,\,1 \leq i \leq 4n-2, \\
       \frac{1}{2}, \text{ if } j = 4n, \,\,1 \leq i \leq 4n, \\
       0, \text{ otherwise. }
   \end{cases}
\end{equation*}
Indeed, the column vectors of $A_{4n}$ are linearly independent and are all in $E_{4n}$. Explicitly,
\begin{equation*}
    A_{4n} =
    \begin{bmatrix}
    2 & -1 & 0 & 0 & \cdots & 0 & 0 & 0 & \frac{1}{2} \\
    0 & 1 & -1 & 0 & \cdots & 0 & 0 & 0 & \frac{1}{2} \\
    0 & 0 & 1 & -1 & \cdots & 0 & 0 & 0 & \frac{1}{2} \\
    0 & 0 & 0 & 1 & \cdots & 0 & 0 & 0 & \frac{1}{2} \\
    \vdots & \vdots & \vdots & \vdots & \ddots & \vdots & \vdots & \vdots & \vdots \\
    0 & 0 & 0 & 0 & \cdots & 1 & -1 & 0 & \frac{1}{2} \\
    0 & 0 & 0 & 0 & \cdots & 0 & 1 & -1 & \frac{1}{2} \\
    0 & 0 & 0 & 0 & \cdots & 0 & 0 & 1 & \frac{1}{2} \\
    0 & 0 & 0 & 0 & \cdots & 0 & 0 & 0 & \frac{1}{2}
\end{bmatrix}.
\end{equation*}
By \cite{conway2013sphere}, $E_{8n}$ is irreducible and even for every $n \geq 1$. Here, even means that the square of the norm of any lattice vector is an even number. Now, Milnor's example consists of the flat tori associated to the lattices $E_{16}$ and $E_8 \times E_8$. Since $E_{16}$ is irreducible and $E_8 \times E_8$ is reducible, it follows from Proposition \ref{prop:reducible_lattice_products} that the flat tori are not isometric. It remains to show that they are isospectral. 
The lattice $E_{16}$ has basis matrix $A_{16}$, and $E_8 \times E_8$ has
\begin{equation*}
    A_{8 \times 8} =
    \begin{bmatrix}
        A_8 & \bm{0} \\
        \bm{0} & A_8
    \end{bmatrix}.
\end{equation*}
Two quadratic forms $P$ and $Q$ associated to $E_{16}$ and $E_8 \times E_8$ are those with associated matrices $A_{16}^TA_{16}$ and $A_{8 \times 8}^TA_{8 \times 8}$, respectively. These are even, positive definite, and have determinant 1, as the reader may verify. In fact, the inverses $(A_{16}^TA_{16})^{-1}$ and $(A_{8 \times 8}^TA_{8 \times 8})^{-1}$ are also even, which gives that $N_{A_{16}^TA_{16}} = N_{A_{8 \times 8}^TA_{8 \times 8}} = 1$. Since $\mu_0(1) = 1$, Theorem \ref{theorem:Isospectrality_certificate} gives that it is enough to show that $P$ and $Q$ have the same multiplicity over 1. That is, we need to show that the equations $P(x) = 1$ and $Q(x) = 1$ have as many distinct solutions in $\Z^{16}$.
Since $E_8$ and $E_{16}$ are even, 
\begin{align*}
    P(x) &= (A_{16}x)^T(A_{16}x) = ||A_{16}x||^2, \\
    Q(x) &= (A_{8\times 8}x)^T(A_{8\times 8}x) = ||A_{8 \times 8}x||^2
\end{align*} 
are always even. So, they both have multiplicity 0 over 1. This completes the proof that the flat tori are isospectral, and we conclude that the spectrum of a Riemannian manifold in general does not determine all of its geometry.

\section{One cannot hear the shape of a Riemannian manifold but can one hear the shape of a drum?}  \label{sec:kac}

Just two years after Milnor's article appeared, Kac published the popular scientific article \cite{kac1966can} titled ``Can one hear the shape of a drum?''  He first heard the problem from Professor Bochner in mathematical terms:  if two planar domains have the same Laplace spectrum, are they congruent in the sense of Euclidean geometry?  The formulation in the title of Kac's article is motivated by Professor Bers's reformulation as ``If you had perfect pitch could you find the shape of a drum?''  We reasonably expect that Kac was inspired by Milnor's article, because Kac describes Milnor's construction on page 3 of \cite{kac1966can}.  Thanks to Kac, it is now common to refer to quantities that are determined by the spectrum, namely spectral invariants, as quantities that can be `heard.'  At the time of his article, Kac knew that one can hear the area and perimeter of a bounded domain in the plane.  However, the answer to Kac's question was unknown, and he wrote that he personally believed that one cannot hear the shape of a tambourine but may well be wrong.  He proposed to investigate how much about the shape of a domain can be inferred from the knowledge of the spectrum and ``to impress upon you the multitude of connections between our problems and various parts of mathematics and physics'' \cite[p. 3]{kac1966can}. Kac was forthcoming that his article was not meant to be a typical rigorous research article, but rather ``more in the nature of a leisurely excursion than an organized tour'' \cite[p.1]{kac1966can}. So, although certain results were heuristically motivated rather than rigorously proven, it turns out that with more modern mathematical techniques, all of Kac's heuristics can be made rigorous.  We therefore take the opportunity here to explain the mathematics of Kac's inspiring article and to show how to make them rigorous.  
First we will need some preliminary notions. 

\subsection{Geometric preliminaries}
In \S\ref{sec:microlocal} , we will need the notion of domains in the plane converging to another domain. To define this, we need a way of measuring the distance between two sets. We will use \textit{Hausdorff distance}. 
Recall that an \textit{isometry} in $\R^2$ is an affine linear mapping from $\R^2$ to itself which corresponds to a composition of rotations, translations, and reflections. If $\Theta$ and $\Omega$ are two sets in $\R^2$, then the Hausdorff distance between $\Theta$ and $\Omega$ is defined as
    \begin{equation*}
        D(\Theta, \Omega) = \inf_{T, S \text{ isometries}} d_H(T(\Theta), S(\Omega)),
    \end{equation*}
    where
    \begin{equation*}
        d_H(\Theta, \Omega) = \max\left\{\sup_{x \in \Theta} \inf_{y \in \Omega} ||x-y||, \,\,\sup_{x \in \Omega} \inf_{y \in \Theta} ||x-y|| \right\}.
    \end{equation*}
    Given a sequence of sets $\{\Omega_k\}_{k=1}^\infty$ in $\R^2$, we say that $\Omega_k$ \textit{converges in Hausdorff} to some set $\Omega \subset \R^2$, written $\Omega_k \to \Omega$ as $k \to \infty$, if $D(\Omega_k, \Omega) \to 0$ as $k \to \infty$. 

For any $x \in \R^2$, $\inf_{y \in \Omega}||x-y||$ is often called the \textit{distance} between $x$ and $\Omega$. The following proposition gives some justification for the above definition. The proof is a straightforward application of the definition and can be found in \cite[p. 6-7]{maardby2023mathematics}.
\begin{proposition} \label{prop:Hausdorf_distance_less_than_epsilon}
    Given two non-empty sets $\Theta, \Omega \subset \R^2$, we have $d_H(\Theta, \Omega) = 0$ if and only if $\overline{\Theta} = \overline{\Omega}$.  If we further assume that $\Theta$ and  $\Omega$ are closed, and $\epsilon > 0$, we have $d_H(\Theta, \Omega) < \epsilon$ if and only if
    \begin{equation} \label{eq:union_balls}
        \begin{split}
            \Omega \subset \bigcup_{x \in \Theta} B_{\epsilon}(x), \\
        \Theta \subset \bigcup_{x \in \Omega} B_{\epsilon}(x),
        \end{split}
    \end{equation}
    where
    \begin{equation*}
        B_\epsilon(x) = \{y \in \R^2 : ||x-y|| < \epsilon\}.
    \end{equation*}
\end{proposition}

In \S\ref{sec:microlocal} we will use the \textit{Euler characteristic} (sometimes called the \textit{Euler number}) of a bounded domain in the plane. If $\Omega \subset \R^2$ is bounded and consists of finitely many polygons, then its Euler characteristic is given by
\begin{equation} \label{eq:euler_char}
        \chi(\Omega) = V - E + F,
    \end{equation}
where $V$ is the number of vertices in $\Omega$, $E$ is the number of edges in $\Omega$, and $F$ is the number of polygons in $\Omega$.

If instead $\Omega$ is a smoothly bounded domain, one computes the Euler characteristic using \textit{triangulations}. Intuitively, a triangulation of $\Omega$ is a collection of ``triangles'' which need not have straight edges. Their union should be $\Omega$ and the intersection of two triangles is either empty, a common edge, or a common vertex. For the precise definition of a triangulation, we refer to \cite[Definition 13.4.1]{pressley2010elementary}. The Euler characterstic of $\Omega$ is then given by the Euler characterstic of the given triangulation. By \cite[Corollary 13.4.6]{pressley2010elementary}, $\chi(\Omega)$ is independent of the choice of triangulation. A \em hole \em of $\Omega$ is a bounded connected component of $\R^2 \setminus \Omega$. The number of holes in a domain is related to the Euler characteristic by the following result. 

\begin{theorem}[\cite{hatcher2005algebraic}, Thm. 2.44] \label{theorem:holes}
    If $\Omega$ has $r$ holes, then $\chi(\Omega) = 1-r$.
\end{theorem}

\subsection{What could Kac hear?} \label{sec:Kac_paper}
In \S \ref{sec:Weyl_approx} we showed how Weyl proved that the area of a domain $\Omega \subset \R^2$ is a spectral invariant. Kac was aware of this result and wanted to find other quantities of $\Omega$ that are spectral invariants. Is perhaps the entire set $\Omega$ a spectral invariant (up to isometries)? In other words, can one hear the shape of a drum? Studying spectral invariants is interesting for several reasons, one of which being its connections to physics. To explore these connections between mathematics and physics, consider a system of $M$ particles in a three-dimensional volume $\Omega$ such that the system is in equilibrium with temperature $T$. We assume that the particles comprise an ideal gas.  The Schrödinger equation for such an ideal gas is a separable equation given by 
\begin{equation} \label{eq:Schrödinger}
    \begin{cases}
        \frac{\hbar^2}{2m}\Delta \phi(x) = -E\phi(x) \text{ in } \Omega, \\ 
        \phi(x) = 0 \text{ on } \partial\Omega.
    \end{cases}
\end{equation}
Above, $\hbar = h/2\pi$ with $h$ being Planck's constant, $m$ is the mass of each particle, and $E > 0$ is some constant. Rearranging this equation is equivalent to the Laplace eigenvalue equation 
\begin{equation} \label{eq:Shrödinger_rearrange}
    \begin{cases}
        \frac{1}{2}\Delta \phi(x) + \lambda \phi(x) = 0 
        \text{ in } \Omega, \\ 
        \phi(x) = 0 \text{ on } \partial\Omega
    \end{cases}
\end{equation}
with $\lambda = mE/\hbar^2$. In particular, Kac chose to include the factor $c^2 = 1/2$ in his eigenvalue equation. One could just as easily get the more standard eigenvalue equation 
\begin{equation} \label{eq:Laplace_c^2=1}  
    \begin{cases}
        \Delta \phi(x) + \Lambda \phi(x) = 0 
        \text{ in } \Omega, \\ 
        \phi(x) = 0 \text{ on } \partial\Omega
    \end{cases}
\end{equation}
by setting $\Lambda = 2\lambda$.

If $(\lambda_n,\phi_n)$ denotes the $n$th normalized eigenpair of \eqref{eq:Shrödinger_rearrange}, then we know from physics Boltzmann statistics \cite[ch. 6]{schroeder1999introduction}, \cite[ch. 6]{henningson2013borja} that the probability of finding specified particles at $x_1,\dots,x_M$ within $dx_1,\dots,dx_M$ is 
\begin{equation} \label{eq:quantum}
    \prod_{k = 1}^M \frac{\sum_{n = 1}^{\infty}e^{-\lambda_n\tau}\phi_n^2(x_k)}{\sum_{n = 1}^{\infty}e^{-\lambda_n\tau}}dx_k, \,\,\, \tau = \frac{\hbar^2}{mkT}.
\end{equation}
If the particles were instead moving uniformly in $\Omega$, then the probability of finding specified particles would be
\begin{equation} \label{eq:classical}
    \frac{dx_1\dots dx_M}{|\Omega|^M}.
\end{equation}  
By the correspondence principle in quantum mechanics, the quantum perspective \eqref{eq:quantum} and classical perspective \eqref{eq:classical} will approach each other as time goes to zero \cite[p. 160-161] {tipler2007modern}. This motivates the study of the short time asymptotic behavior of the heat trace
\begin{equation*}
    \sum_{n = 1}^\infty e^{-\lambda_n t}, \,\, t \to 0,
\end{equation*}
which will (hopefully) reveal geometric properties of $\Omega$.  Kac explained that one of the goals of his article was to obtain the first three terms in the short time asymptotic expansion of the heat trace in the case where $\Omega$ is a smoothly bounded convex domain in the plane, namely to prove that 
\begin{equation} \label{eq:heattrace_threeterms}
    \sum_{n = 1}^\infty e^{-\lambda_n t} \sim \frac{|\Omega|}{2\pi t} - \frac{|\partial \Omega|}{4\sqrt{2\pi t}} + \frac{\chi(\Omega)}{6}, \,\,\, t \to 0.
\end{equation}
This would show that in addition to the area and the perimeter, the Euler characteristic is also a spectral invariant.  To rigorously obtain this asymptotic expansion, we require further particulars concerning the heat equation, its fundamental solution, and its trace. 

\subsubsection{Hot preliminaries} \label{sec:heatkernel}
Here we summarize the definitions and results needed for 
different objects connected to the heat equation, including the heat kernel and the heat trace. 
More details on these topics and the proofs of the statements below can be found in \cite
{albin2017linear}. The heat equation for a bounded domain in $\R^n$ with the Dirichlet boundary condition and continuous initial data is 
\begin{equation} \label{eq:heat}
        \begin{cases}
            \frac{\partial }{\partial t}u(x,t) - \Delta u(x, t) = 0, \,\, x \in \Omega, \,\, t > 0, \\
            u(\cdot,t) \in H_0^1(\Omega) \cap H^2(\Omega), \,\, t > 0, \\
            u(x,0) = f(x) \in \mathscr{C}_c^0(\overline{\Omega}).
        \end{cases}
\end{equation}
We have the following result which summarizes the existence and properties of the fundamental solution to the heat equation.  
\begin{theorem}
    There exists a unique map 
$e^{-t\Delta} : \mathscr{C}^0(\overline{\Omega}) \to \mathscr{C}^\infty(\overline{\Omega} \times \R^+)$, called \textit{heat operator}, which satisfies  
\begin{equation*} 
        \begin{cases}
            \left(\frac{\partial }{\partial t} - \Delta\right)e^{-t\Delta}(f)(x,t) = 0, \,\, x \in \Omega, \,\, t > 0, \\
            e^{-t\Delta}(f)(x,0) = f(x) \in \mathscr{C}_c^0(\overline{\Omega}).
        \end{cases}
\end{equation*}
Moreover, there exists a map $h : \Omega \times \Omega \times \R^+ \to \R$ such that $h$ is the Schwartz kernel of $e^{-t\Delta}$, meaning that
\begin{equation*}
    e^{-t\Delta}(f)(x,t) = \int_\Omega h(x,y,t)f(y) dy, \,\, x \in \Omega, \,\, t > 0.
\end{equation*}
This Schwartz kernel, $h$, is known as the \em heat kernel \em and satisfies 
    \begin{align}
        h(x,y,t) &= h(y,x,t), \,\, x,y \in \Omega, \,\, t > 0, \label{eq:heat_symm} \\
        \left(\frac{\partial}{\partial t} - \Delta\right)h(x,y,t) &= 0 \text{ for } \Delta \text{ acting in either } x \text{ or } y, \,\, t > 0, \label{eq:heat_eq_ker} \\
         \lim_{t \to 0} \int_\Omega h(x,y,t) f(y) dy &= f(x). \label{eq:heat_ker_ic} 
    \end{align}
\end{theorem}

The heat kernel for a bounded domain can be expressed in terms of the eigenfunctions and eigenvalues of the Laplacian.     

\begin{proposition} \label{prop:heatkernel_eigenfunctions}
    If $\Omega$ is bounded, then
    \begin{equation} \label{eq:heatkernel_eigenfunctions}
        h(x,y,t) = \sum_{n = 1}^\infty e^{-\lambda_nt}\phi_n(x)\overline{\phi_n(y)},
    \end{equation}
    where $\phi_n$ are normalized eigenfunctions with corresponding eigenvalues $\lambda_n$, so that 
    \begin{equation*} 
            \begin{cases}
                \Delta \phi_n(x) = \lambda_n\phi_n(x), \,\, x \in \Omega, \\
                \phi_n \in H_0^1(\Omega) \cap H^2(\Omega).
            \end{cases}
    \end{equation*}
\end{proposition}
The \textit{heat trace} is defined as 
    \begin{equation*}
        \int_\Omega h(x,x,t) dx = \sum_{n = 1}^\infty e^{-\lambda_nt} \int_\Omega|\phi_n(x)|^2 dx = \sum_{n = 1}^\infty e^{-\lambda_nt}.
    \end{equation*}
Note that integrating termwise is justified by the monotone convergence theorem (see \cite[p. 51]{folland1999real}).  The heat trace is a spectral invariant, and therefore any information it contains can be heard.  To glean information from the heat trace, one can use local models to approximate the heat kernel for the domain $\Omega$.  On the interior of the domain, at some positive distance from the boundary, the heat kernel resembles the so-called free heat kernel for Euclidean space, at least for small time.   It is therefore convenient to explicitly compute this heat kernel for Euclidean space $\R^n$.  We include a proof to impress upon readers the properties of the Euclidean heat kernel.  In particular, away from the diagonal, that is for $x \neq y$, this heat kernel tends to $0$ rapidly as $t \downarrow 0$, whereas for $x=y$, it blows up at the order of $t^{-n/2}$.  Consequently, the initial condition \eqref{eq:heat_ker_ic} is obtained as a limit rather than simply substituting $t=0$ into the expression.  

\begin{proposition} \label{prop:heatkernel_R^n}
    If $\Omega = \R^n$, then the unique element of $L^2(\Omega \times \Omega \times \R^+)$ that satisfies \eqref{eq:heat_symm}, \eqref{eq:heat_eq_ker}, and \eqref{eq:heat_ker_ic} is 
    \begin{equation*}
            h_{\R^n}(x,y,t) = \frac{1}{(4\pi t)^{n/2}}e^{-||x-y||^2/4t}, \,\, x,y \in \R^n, \,\, t > 0.
    \end{equation*}
    \begin{proof}
    One can check that this function satisfies \eqref{eq:heat_symm} and \eqref{eq:heat_eq_ker} by direct computation.  It is not immediately obvious, however, that 
    \begin{equation*}
        \lim_{t \to 0} \int_{\R^n} \frac{1}{(4\pi t)^{n/2}}e^{-||x-y||^2/4t} f(y) dy = f(x). 
    \end{equation*}
    To prove this, fix $x \in \R^n$, let $\epsilon > 0$, and choose $\delta > 0$ so that $|f(y)-f(x)| < \epsilon$ whenever $||y-x|| < \delta$. This is possible because we have assumed that $f$ is continuous and has compact support.  It is a straightforward computation to show that
    \begin{equation*}
        \int_{\R^n} \frac{1}{(4\pi t)^{n/2}}e^{-||x-y||^2/4t} dx = 1.
    \end{equation*}
    This gives that
    \begin{align*}
        &\left|\int_{\R^n} \frac{1}{(4\pi t)^{n/2}}e^{-||x-y||^2/4t}f(y) dy - f(x) \right| = \left|\int_{\R^n} \frac{1}{(4\pi t)^{n/2}}e^{-||x-y||^2/4t}(f(y)-f(x)) dy \right| \\
        \leq &\int_{\R^n} \frac{1}{(4\pi t)^{n/2}}e^{-||x-y||^2/4t}|f(y)-f(x)| dy 
    \end{align*}
    \begin{equation} \label{eq:IJ}
        \begin{split}
            = &\int_{B_\delta(x)} \frac{1}{(4\pi t)^{n/2}}e^{-||x-y||^2/4t}|f(y)-f(x)| dy \\
            + &\int_{\R^n \backslash B_\delta(x)} \frac{1}{(4\pi t)^{n/2}}e^{-||x-y||^2/4t}|f(y)-f(x)| dy,
        \end{split}
    \end{equation}
    where, as before,
    \begin{equation*}
        B_\delta(x) = \{y \in \R^n : ||y-x|| < \delta\}.
    \end{equation*}
    Since $|f(y) - f(x)| < \epsilon$ for all $y \in B_\delta(x)$, the first term in \eqref{eq:IJ} can be estimated as
    \begin{equation*}
        \int_{B_\delta(x)} \frac{1}{(4\pi t)^{n/2}}e^{-||x-y||^2/4t}|f(y)-f(x)| dy \leq \epsilon\int_{B_\delta(x)} \frac{1}{(4\pi t)^{n/2}}e^{-||x-y||^2/4t} dy \leq \epsilon.
    \end{equation*}
    For the second term in \eqref{eq:IJ}, recall that $f$ has compact support (in particular, $f$ is bounded) and write
    \begin{equation*}
        C = \frac{2\max_{x \in \R^n}|f(x)|}{(4\pi)^{n/2}}.
    \end{equation*}
    Then
    \begin{equation*}
        \int_{\R^n \backslash B_\delta(x)} \frac{1}{(4\pi t)^{n/2}}e^{-||x-y||^2/4t}|f(y)-f(x)| dy \leq \frac{C}{t^{n/2}} \int_{\R^n \backslash B_\delta(x)} e^{-||x-y||^2/4t} dy,
    \end{equation*}
    which, after change of variables, can be written as
    \begin{equation*}
        C\int_{\R^n \backslash B_{\delta/\sqrt{t}}(0)} e^{-||y||^2/4} dy.
    \end{equation*}
    This goes to 0 as $t \to 0$. In particular, the second term in \eqref{eq:IJ} will be less than $\epsilon$ for $t$ sufficiently close to 0. Then
    \begin{equation*}
        \left|\int_{\R^n} \frac{1}{(4\pi t)^{n/2}}e^{-||x-y||^2/4t} f(y) dy - f(x) \right| < 2\epsilon,
    \end{equation*}
    which completes the proof of \eqref{eq:heat_ker_ic}.  
    
    To prove uniqueness, suppose $g$ also satisfies \eqref{eq:heat_symm}, \eqref{eq:heat_eq_ker}, and \eqref{eq:heat_ker_ic}. Then we fix $x$ and $t$ and note that 
    \begin{equation*}
        \int_{\R^n} (h-g)(x,y,t)f(y) dy = 0
    \end{equation*}
    for any continuous and compactly supported function $f$. Now let $\epsilon > 0$. Since $h$ and $g$ are in $L^2$, there is an $R > 0$ such that
    \begin{align*}
        &\int_{\R^n \backslash B_R(0)} h(x,y,t)^2 dy < \epsilon, \\
        &\int_{\R^n \backslash B_R(0)} g(x,y,t)^2 dy < \epsilon. 
    \end{align*}
    Then we have
    \begin{align*}
        &\int_{\R^n} (h-g)^2(x,y,t) dy = \int_{B_R(0)} (h-g)^2(x,y,t) dy + \int_{\R^n\backslash B_R(0)} (h-g)^2(x,y,t) dy \\
        \leq &\int_{B_R(0)} (h-g)^2(x,y,t) dy + 2\int_{\R^n\backslash B_R(0)} h(x,y,t)^2 dy + 2\int_{\R^n\backslash B_R(0)} g(x,y,t)^2 dy \\
        < &\int_{B_R(0)} (h-g)^2(x,y,t) dy + 4\epsilon,
    \end{align*}
    where we used the inequality $(a-b)^2 \leq 2a^2 + 2b^2$. Moreover, since the collection of continuous functions is dense in $L^2$ for bounded domains, there is a continuous function $f : \R^n \to \R$ such that $|(h-g)(x,y,t) - f(y)| < \epsilon$ for all $y \in B_R(0)$. We may assume that the support of $f$ is contained in $B_R(0)$. This gives
    \begin{align*}
        &\int_{B_R(0)} (h-g)^2(x,y,t) dy \leq \int_{B_R(0)} (h-g)(x,y,t)f(y) dy + \epsilon\int_{B_R(0)} (h-g)(x,y,t) dy \\
         = &\int_{\R^n} (h-g)(x,y,t)f(y) dy + \epsilon\int_{B_R(0)} (h-g)(x,y,t) dy = \epsilon\int_{B_R(0)} (h-g)(x,y,t) dy.
    \end{align*}
    Since $h$ and $g$ are in $L^2$, it follows from Cauchy-Schwarz' inequality that the last integral is bounded. Thus, there is a constant $M > 0$, independent of $\epsilon > 0$, such that 
    \begin{equation*}
        \int_{\R^n} (h-g)^2(x,y,t) dy < \epsilon M.
    \end{equation*}
    Since $\epsilon > 0$ was arbitrary, it follows that $h$ and $g$ are equal as elements of $L^2$ with respect to the variable $y$, for each fixed $x$ and $t$. Finally, since $x$ and $t$ were arbitrary, we conclude that $h$ and $g$ are the same element of $L^2(\Omega \times \Omega \times \R^+)$. 
    \end{proof} 
\end{proposition}
If $\Omega = \R_+^n = \{x = (x_1,\dots,x_n) \in \R^n : x_n \geq 0\}$, 
then it can similarly be shown, using the method of images \cite{gaponenko1969method}, that the heat kernel with Dirichlet boundary conditions at $x_n = 0$ is
\begin{equation} \label{eq:heatkernel_halfspace}
    h_{\R_+^n}(x,y,t) = \frac{e^{-||x-y||^2/4t} - e^{-||x+y||^2/4t}}{(4\pi t)^{n/2}}, \,\, x,y \in \R^n.
\end{equation}
This heat kernel can be used to approximate that of a smoothly bounded domain near its edges, or a polygonal domain near its edges avoiding the corners.  We now proceed to investigate the behavior of the heat trace as $t \downarrow 0$.

\subsubsection{Kac's principle of not feeling the boundary and the leading order term} \label{sec:t^-1rigorous}
Kac used heuristics from physics to compare the heat kernel of a bounded domain with that of $\R^2$ to obtain the first term in the asymptotic expansion of the heat trace. He argued that particles in the interior of the domain should not yet have felt the effect of the boundary if the time is very small, and therefore the heat kernel for the domain should should be asymptotically equal to that of $\R^2$ as the time approaches zero. In this way, the integral of the Euclidean heat kernel along the diagonal $(x=y)$ over the domain should give the leading order asymptotic term in the heat trace. In fact, Kac also made this argument almost fully rigorous using Weyl's law (see Theorem \ref{theorem:Weyl}) in \cite[\S 8]{kac1966can}.  We include here a fully rigorous treatment.

Note that if we insert $\lambda_n$ into \eqref{eq:Weyl}, we get $\lambda_n \sim \frac{4\pi n}{|\Omega|}, \, n \to \infty$. This implies that we, for any given $\epsilon > 0$, can choose an $N \in \mathbb{N}$ such that
\begin{equation} \label{eq:epsN}
    \frac{4\pi n}{|\Omega|}(1-\epsilon) \leq \lambda_n \leq \frac{4\pi n}{|\Omega|}(1+\epsilon)
\end{equation}
for $n \geq N$. Then
\begin{equation} \label{eq:sum_epsN}
    \sum_{n = N}^\infty e^{-4\pi nt(1+\epsilon)/|\Omega|} \leq \sum_{n = N}^\infty e^{-\lambda_n t} \leq \sum_{n = N}^\infty e^{-4\pi nt(1-\epsilon)/|\Omega|}
\end{equation}
for $t > 0$. Summing the geometric series, we obtain
\begin{equation*}
    \sum_{n = N}^\infty e^{-4\pi nt(1\pm\epsilon)/|\Omega|} = \frac{e^{-4\pi t(N-1)(1\pm\epsilon)/|\Omega|}}{e^{4\pi t(1\pm\epsilon)/|\Omega|} - 1}.
\end{equation*}
We have thus shown that
\begin{equation*}
    \sum_{n=1}^{N-1}e^{-\lambda_n t} + \frac{e^{-4\pi t(N-1)(1+\epsilon)/|\Omega|}}{e^{4\pi t(1+\epsilon)/|\Omega|} - 1} \leq \sum_{n=1}^\infty e^{-\lambda_n t} \leq \sum_{n=1}^{N-1} e^{-\lambda_n t} + \frac{e^{-4\pi t(N-1)(1-\epsilon)/|\Omega|}}{e^{4\pi t(1-\epsilon)/|\Omega|} - 1}.
\end{equation*}
In particular,
\begin{align*}
    \begin{split}
        &\liminf_{t \to 0}\left(t\sum_{n=1}^{N-1}e^{-\lambda_n t} + t\frac{e^{-4\pi t(N-1)(1+\epsilon)/|\Omega|}}{e^{4\pi t(1+\epsilon)/|\Omega|} - 1}\right) \leq \liminf_{t\to 0} t\sum_{n=1}^\infty e^{-\lambda_n t} \\ \leq &\limsup_{t \to 0}t\sum_{n=1}^\infty e^{-\lambda_n t} \leq \limsup_{t \to 0} \left(t\sum_{n=1}^{N-1} e^{-\lambda_n t} + t\frac{e^{-4\pi t(N-1)(1-\epsilon)/|\Omega|}}{e^{4\pi t(1-\epsilon)/|\Omega|} - 1}\right),
    \end{split}
\end{align*}
and since
\begin{equation*}
    \lim_{t\to 0} t\frac{e^{-4\pi t(N-1)(1\pm\epsilon)/|\Omega|}}{e^{4\pi t(1\pm\epsilon)/|\Omega|} - 1} = \frac{|\Omega|}{4\pi(1\pm \epsilon)},
\end{equation*}
we get
\begin{equation*}
    \frac{|\Omega|}{4\pi(1+ \epsilon)} \leq \liminf_{t\to 0} t\sum_{n=1}^\infty e^{-\lambda_n t} \leq \limsup_{t\to 0} t\sum_{n=1}^\infty e^{-\lambda_n t} \leq \frac{|\Omega|}{4\pi(1- \epsilon)}.
\end{equation*}
By letting $\epsilon \to 0$, it follows that 
\begin{equation*}
    \lim_{t \to 0} t\sum_{n=1}^\infty e^{-\lambda_n t}
\end{equation*}
exists and equals $|\Omega|/4\pi$, which shows that 
\begin{equation*}
    \sum_{n=1}^\infty e^{-\lambda_n t} \sim \frac{|\Omega|}{4\pi t}, \,\,\,t \to 0.
\end{equation*}
If we, like Kac, would have chosen to include the factor $c^2 = 1/2$ as in \eqref{eq:Shrödinger_rearrange}, then by Weyl's law we would instead have $\lambda_n \sim \frac{2\pi n}{|\Omega|}$, $n \to \infty$, from which a similar calculation as above would have given
\begin{equation*}
    \sum_{n=1}^\infty e^{-\lambda_n t} \sim \frac{|\Omega|}{2\pi t}, \,\,\,t \to 0.
\end{equation*}

\subsubsection{Kac's principle of \textbf{feeling the boundary} and the $t^{-1/2}$-term} \label{sec:locality_principle}
To obtain the second term in the heat trace expansion, Kac turned his attention to polygonal domains.  Later, he would use polygonal domains to approximate smoothly bounded domains and calculate the limits of the terms in the heat trace expansion as the polygonal domains converge in the Hausdorff distance to a smoothly bounded domain.  We expect that he turned to polygonal domains because he had explicit expressions for the heat kernels for half-spaces and circular sectors, which he called wedges.  Kac argued that particles shouldn't be able to distinguish between a neighborhood of a straight edge of a polygonal domain and a neighborhood of a straight edge of a half space, hence the heat kernels should be asymptotically equal.  Similarly, near a vertex, particles should not be able to distinguish between a neighborhood of a vertex in a polygonal domain and a neighborhood of a vertex in a circular sector (wedge), hence the heat kernels should be asymptotically equal.  We call this \textit{Kac's principle of feeling the boundary.}  Although he did not rigorously prove it, his heuristics are correct, and the following formulation can be rigorously proven.  To state this result, for a function $f : (0,\infty) \to \R$, we write $f(t) = \mathcal{O}(t^\infty)$, $t \to 0$, to indicate that $f$ decays faster than any polynomial as $t \to 0$, i.e. for all $N \in \N$,
\begin{equation*}
    \lim_{t \to 0} \frac{f(t)}{t^N} = 0.
\end{equation*}
With this notation, Kac's principle of feeling the boundary says the following.

\begin{theorem}[\cite{van1988heat}, \cite{nursultanov2019hear}, Kac's principle of feeling the boundary] \label{theorem:locality_principle}
    Let $\Omega \subset \R^2$ be a polygonal domain. Let $h_\Omega$ denote the heat kernel for the Laplacian on $\Omega$ with Dirichlet boundary conditions. Then, for $S = S_\theta$, an infinite sector of opening angle $\theta$ and with heat kernel $h_{S(\theta)}$ for Dirichlet boundary conditions, and for any corner of $\Omega$ with opening angle $\theta$, there is a neighborhood $U_\theta \subset \overline \Omega$ such  that 
    \begin{equation*}
        |h_\Omega(x,y,t) - h_{S(\theta)}(x,y,t)| = \mathcal{O}(t^\infty) \text{ uniformly in } (x,y) \in U_\theta \times U_\theta.
    \end{equation*}
    Moreover, for any set $U_e \subset \overline \Omega$ with positive distance to all corners of $\Omega$, we have
    \begin{equation*}
        |h_\Omega(x,y,t) - h_{\R_+^2}(x,y,t)| = \mathcal{O}(t^\infty) \text{ uniformly in } (x,y) \in U_e \times U_e. 
    \end{equation*} 
        We note that both of these estimates hold all the way up to the boundary of $\Omega$. 
\end{theorem} 
This shows that to obtain the asymptotic expansion of the heat trace in the polygon $\Omega$, it is enough to replace the actual heat kernel $h_\Omega$ with either $h_{S(\theta)}$ if we are close to a corner with angle $\theta$, or with $h_{\R_+^2}$ otherwise. This is indeed what Kac did, so in essence he was correct, but he did not provide a full rigorous proof.  

We have the expression for $h_{\R_+ ^2}$, thus to continue we require the heat kernel $h_{S(\theta)}$. Kac's expression for the heat kernel on an infinite sector appears with a reference to Carslaw \cite{carslaw1910green} and no further details.  It requires some effort to obtain Kac's expression.  These missing details have been computed in \cite[\S6]{aldana2018polyakov}; an even more pedagogical presentation is contained in \cite[p. 40-44]{maardby2023mathematics}. In polar coordinates $(r, \phi)$ for a circular sector of opening angle $\theta$ and infinite radius, as shown in Figure \ref{fig:infinite_sector}, the heat kernel restricted to the diagonal is for $0 < \phi < \theta - \pi/2$ 
\begin{figure}
\centering
\begin{tikzpicture}
    \draw[-] (3,3) -- (8,3); 
    \draw[-] (3,3) -- (0,8);
    \draw[-] [dashed] (3,3) -- (5,6) node[midway, above = 3pt, left = 1pt]{$r$};
    \draw (3,3) circle[radius=2pt]; and \fill (3,3) node[left = 4pt, above = -16pt]{$O$} circle[radius=2pt];
    \draw (3,3) circle[radius=2pt]; and \fill (3,3) node[right = 15pt, above = 9pt]{$\theta$} circle[radius=2pt];
    \draw (3.85,3) circle[radius=0pt]; and \fill (3.85,3) node[right = 0pt, above = 5pt]{$\phi$} circle[radius=0pt];
    \draw (5,6) circle[radius=2pt]; and \fill (5,6) node[right = 5pt, above = 4pt]{$(r,\phi)$} circle[radius=2pt];
    \draw (3.5,3) arc (10:115:0.5); 
    \draw (3.8,3) arc (10:60:0.85); 
\end{tikzpicture}
\caption{An infinite sector with opening angle $\theta \in (\pi/2, \pi)$.}
\label{fig:infinite_sector}
\end{figure}
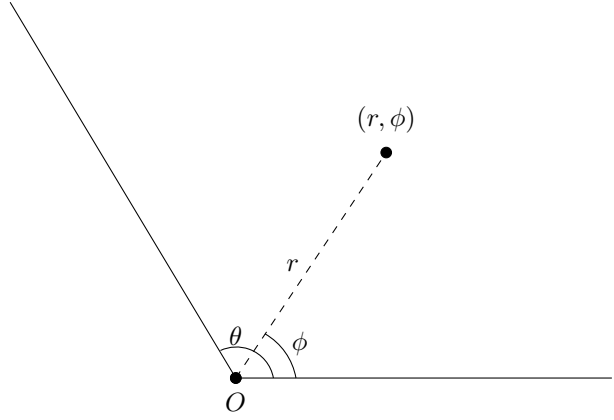
\begin{equation} \label{eq:heatkernel_sector1}
    \begin{split}
        h_{S(\theta)}(r,\phi,r,\phi,&t) = \frac{1}{4\pi t} - \frac{e^{-r^2(1-\cos(2\phi))/2t}}{4\pi t} \\
        &-\sin\left(\frac{\pi^2}{\theta}\right) \frac{e^{-r^2/2t}}{8\pi\theta t} \int_{-\infty}^\infty \frac{e^{-r^2\cosh(s)/2t}}{\cosh(\pi s/\theta) - \cos(\pi^2/\theta)} ds \\
        &+\sin\left(\frac{\pi^2}{\theta}\right) \frac{e^{-r^2/2t}}{8\pi\theta t} \int_{-\infty}^\infty \frac{e^{-r^2\cosh(s)/2t}}{\cosh(\pi s/\theta + 2\pi i\phi/\theta) - \cos(\pi^2/\theta)} ds.
    \end{split}
\end{equation}
For $\theta - \pi/2 < \phi < \pi/2$, 
\begin{equation} \label{eq:heatkernel_sector2}
    \begin{split}
        h_{S(\theta)}(r,\phi,r,\phi,&t) = \frac{1}{4\pi t} - \frac{e^{-r^2(1-\cos(2\phi))/2t}}{4\pi t} - \frac{e^{-r^2(1-\cos(2(\theta-\phi))/2t}}{4\pi t} \\
        &-\sin\left(\frac{\pi^2}{\theta}\right) \frac{e^{-r^2/2t}}{8\pi\theta t} \int_{-\infty}^\infty \frac{e^{-r^2\cosh(s)/2t}}{\cosh(\pi s/\theta) - \cos(\pi^2/\theta)} ds \\
        &+\sin\left(\frac{\pi^2}{\theta}\right) \frac{e^{-r^2/2t}}{8\pi\theta t} \int_{-\infty}^\infty \frac{e^{-r^2\cosh(s)/2t}}{\cosh(\pi s/\theta + 2\pi i\phi/\theta) - \cos(\pi^2/\theta)} ds,
    \end{split}
\end{equation}
and for $\pi/2 < \phi < \theta$,
\begin{equation} \label{eq:heatkernel_sector3}
    \begin{split}
        h_{S(\theta)}(r,\phi,r,\phi,&t) = \frac{1}{4\pi t} - \frac{e^{-r^2(1-\cos(2(\theta-\phi))/2t}}{4\pi t} \\
        &-\sin\left(\frac{\pi^2}{\theta}\right) \frac{e^{-r^2/2t}}{8\pi\theta t} \int_{-\infty}^\infty \frac{e^{-r^2\cosh(s)/2t}}{\cosh(\pi s/\theta) - \cos(\pi^2/\theta)} ds \\
        &+\sin\left(\frac{\pi^2}{\theta}\right) \frac{e^{-r^2/2t}}{8\pi\theta t} \int_{-\infty}^\infty \frac{e^{-r^2\cosh(s)/2t}}{\cosh(\pi s/\theta + 2\pi i\phi/\theta) - \cos(\pi^2/\theta)} ds.
    \end{split}
\end{equation}
With our explicit heat kernel expressions \eqref{eq:heatkernel_halfspace}, and \eqref{eq:heatkernel_sector1}, \eqref{eq:heatkernel_sector2}, \eqref{eq:heatkernel_sector3}, we proceed with our heat trace calculations.  
Suppose $\Omega$ has $N$ vertices $V_1,\dots,V_N$ 
with angles $\theta_1,\dots,\theta_N$. For each $i = 1,\dots,N$, choose $U_{\theta_i}$ according to Theorem \ref{theorem:locality_principle}, and choose $R_i > 0$ such that the sector $S(R_i, \theta_i)$ with radius $R_i$ and angle $\theta_i$ is contained in $U_{\theta_i}$. Then each $S(R_i, \theta_i)$ is contained in the infinite sector $S(\theta_i)$. Assuming, without loss of generality, that the sectors $S(R_i, \theta_i)$ are disjoint, the heat trace becomes
\begin{equation*}
    \int_\Omega h_\Omega(x,x,t) dx \sim \sum_{i = 1}^N \int_{S(R_i, \theta_i)} h_{S(\theta_i)}(x,x,t) dx + \int_{\Omega \backslash \bigcup_{i=1}^N S(R_i, \theta_i)} h_{\R_+^2}(x,x,t) dx, 
\end{equation*}
as $t \downarrow 0$.  
To compute $\int_{S(R_i, \theta_i)} h_{S(\theta_i)}(x,x,t) dx$, we note that since $S(R_i, \theta_i)$ has finite radius $R_i$, 
\begin{align*}
    &\int_{S(R_i, \theta_i)} h_{S(\theta_i)}(x,x,t) dx = \int_0^{\theta_i}\int_0^{R_i} h_{S(\theta_i)}(r, \phi, r, \phi, t) rdrd\phi \\
     = &\int_0^{\theta_i}\int_0^{R_i} \frac{1}{4\pi t} rdrd\phi  -  \int_0 ^{\pi/2}\int_0 ^{R_i}  \frac{e^{-r^2(1-\cos(2\phi))/2t}}{4\pi t} rdr d\phi \\ 
     & - \int_{\theta_i-\pi/2} ^{\theta_i} \int_0 ^{R_i} \frac{e^{-r^2(1-\cos(2(\theta_i-\phi))/2t}}{4\pi t} r dr d\phi \\
        - &\int_0^{\theta_i}\int_0^{R_i} \sin\left(\frac{\pi^2}{\theta_i}\right) \frac{e^{-r^2/2t}}{8\pi\theta_i t} \int_{-\infty}^\infty \frac{e^{-r^2\cosh(s)/2t}}{\cosh(\pi s/\theta_i) - \cos(\pi^2/\theta_i)} rds drd\phi \\
        + &\int_0^{\theta_i}\int_0^{R_i} \sin\left(\frac{\pi^2}{\theta_i}\right) \frac{e^{-r^2/2t}}{8\pi\theta_i t} \int_{-\infty}^\infty \frac{e^{-r^2\cosh(s)/2t}}{\cosh(\pi s/\theta_i + 2\pi i\phi/\theta_i) - \cos(\pi^2/\theta_i)} rds drd\phi.
\end{align*}

Making a change of variables we calculate 
\begin{align*} &-  \int_0 ^{\pi/2}\int_0 ^{R_i}  \frac{e^{-r^2(1-\cos(2\phi))/2t}}{4\pi t} rdr d\phi  - \int_{\theta_i-\pi/2} ^{\theta_i} \int_0 ^{R_i} \frac{e^{-r^2(1-\cos(2(\theta_i-\phi))/2t}}{4\pi t} r dr d\phi \\ &= - \int_0 ^{\pi/2} \int_0 ^{R_i} \frac{e^{-r^2(1-\cos(2\phi))/2t}}{2\pi t} r dr d\phi. 
\end{align*}

To simplify this, we use the result from \cite[10.32.1]{dlmf} together with a change of variables 
\[ I_0 (z) = \frac 1 \pi \int_0 ^\pi e^{\pm z \cos(\theta)} d\theta = \frac 2 \pi \int_0 ^{\pi/2} e^{\pm z \cos(2\phi)} d\phi. \]
    Here $I_0$ is the modified Bessel function of order 0:
    \begin{equation*}
        I_n(x) = \left(\frac{x}{2}\right)^n\sum_{k = 0}^\infty \frac{(\frac{1}{4}x^2)^k}{k!\Gamma(n+k+1)}, \,\, x \geq 0, \,\, n \geq 0.
    \end{equation*}
We use this to calculate 
\[ - \int_0 ^{\pi/2}  \frac{e^{-r^2(1-\cos(2\phi))/2t}}{2\pi t}  d\phi = - \frac{e^{-r^2/(2t)} I_0 (r^2/(2t))}{4t}\]
and therewith 
\[ - \int_0 ^{R_i} \frac{e^{-r^2/(2t)} I_0 (r^2/(2t))}{4t} r dr = - \frac 1 4 \int_0^{R_i^2/2t} e^{-u}I_0(u) du. \] 
Using identities contained in \cite{watson1922treatise} we calculate that 
\begin{equation*}
    \frac{d}{du} \bigg(ue^{-u}\big(I_0(u) + I_1(u)\big)\bigg) = e^{-u}I_0(u)
\end{equation*}
thus 
\begin{equation*}
- \frac 1 4 \int_0^{R_i^2/2t} e^{-u}I_0(u) du = 
    - \frac{1}{8t} R_i^2e^{-R_i^2/2t} \big(I_0(R_i^2/2t) + I_1(R_i^2/2t)\big).
\end{equation*}
The area of the sector $|S(R_i, \theta_i)| = \frac{R_i^2 \theta_i}{2}$, so we therefore obtain 
\begin{align*}
    &\int_{S(R_i, \theta_i)} h_{S(\theta_i)}(x,x,t) dx = \int_0^{\theta_i}\int_0^{R_i} h_{S(\theta_i)}(r, \phi, r, \phi, t) rdrd\phi \\
     = & \frac{|S(R_i, \theta_i)|}{4\pi t}  - \frac{1}{8t} R_i^2e^{-R_i^2/2t} \big(I_0(R_i^2/2t) + I_1(R_i^2/2t)\big) \\ 
        - &\int_0^{\theta_i}\int_0^{R_i} \sin\left(\frac{\pi^2}{\theta_i}\right) \frac{e^{-r^2/2t}}{8\pi\theta_i t} \int_{-\infty}^\infty \frac{e^{-r^2\cosh(s)/2t}}{\cosh(\pi s/\theta_i) - \cos(\pi^2/\theta_i)} rds drd\phi \\
        + &\int_0^{\theta_i}\int_0^{R_i} \sin\left(\frac{\pi^2}{\theta_i}\right) \frac{e^{-r^2/2t}}{8\pi\theta_i t} \int_{-\infty}^\infty \frac{e^{-r^2\cosh(s)/2t}}{\cosh(\pi s/\theta_i + 2\pi i\phi/\theta_i) - \cos(\pi^2/\theta_i)} rds drd\phi.
\end{align*}

While the second integral
\begin{equation*}
    \int_0^{\theta_i}\int_0^{R_i} \sin\left(\frac{\pi^2}{\theta_i}\right) \frac{e^{-r^2/2t}}{8\pi\theta_i t} \int_{-\infty}^\infty \frac{e^{-r^2\cosh(s)/2t}}{\cosh(\pi s/\theta_i + 2\pi i\phi/\theta_i) - \cos(\pi^2/\theta_i)} rds drd\phi = 0,
\end{equation*}
the first integral becomes
\begin{align*}
    &\int_0^{\theta_i}\int_0^{R_i} -\sin\left(\frac{\pi^2}{\theta_i}\right) \frac{e^{-r^2/2t}}{8\pi\theta_it} \int_{-\infty}^\infty \frac{re^{-r^2\cosh(s)/2t}}{\cosh(\pi s/\theta_i) - \cos(\pi^2/\theta_i)} ds drd\phi \\
    &= -\frac{1}{8\pi}\sin\left(\frac{\pi^2}{\theta_i}\right) \int_{-\infty}^\infty \frac{1 - e^{-R_i^2(1+\cosh(s))/2t}}{(1+\cosh(s))(\cosh(\pi s/\theta_i) - \cos(\pi^2/\theta_i))} ds.
\end{align*}
However, since $e^{-R_i^2(1+\cosh(s))/2t} \leq e^{-R_i^2/t} = \mathcal{O}(t^\infty), \,\, t \to 0$, it follows that
\begin{equation*}
    \int_{-\infty}^\infty \frac{e^{-R_i^2(1+\cosh(s))/2t}}{(1+\cosh(s))(\cosh(\pi s/\theta_i) - \cos(\pi^2/\theta_i))} ds  = \mathcal{O}(t^\infty), \,\, t \to 0.
\end{equation*}
This gives  
\begin{align*}
    &\int_0^{\theta_i}\int_0^{R_i} -\sin\left(\frac{\pi^2}{\theta_i}\right) \frac{e^{-r^2/2t}}{8\pi\theta_it} \int_{-\infty}^\infty \frac{re^{-r^2\cosh(s)/2t}}{\cosh(\pi s/\theta_i) - \cos(\pi^2/\theta_i)} ds drd\phi \\
    &\sim -\frac{1}{8\pi} \sin\left(\frac{\pi^2}{\theta_i}\right) \int_{-\infty}^\infty \frac{1}{(1+\cosh(s))(\cosh(\pi s/\theta_i) - \cos(\pi^2/\theta_i))} ds, \,\, t \to 0.
\end{align*}
Now we use that 
\begin{align*}
    &I_0(x) \sim \frac{e^x}{\sqrt{2\pi x}}, \,\, x \to \infty, \\
    &I_1(x) \sim \frac{e^x}{\sqrt{2\pi x}}, \,\, x \to \infty
\end{align*}
(see \cite{watson1922treatise}). In particular,
\begin{align*}
    I_0(R_i^2/2t) \sim \frac{1}{R_i}\sqrt{\frac{t}{\pi}} e^{R_i^2/2t}, \,\, t \to 0, \\
    I_1(R_i^2/2t) \sim \frac{1}{R_i}\sqrt{\frac{t}{\pi}} e^{R_i^2/2t}, \,\, t \to 0.
\end{align*}
This gives that the total contribution from the sectors becomes asymptotically equal to
\begin{align*}
    &\frac{1}{4\pi t} \sum_{i = 1}^N |S(R_i, \theta_i)| - \frac{1}{4\sqrt{\pi t}}\sum_{i=1}^N R_i \\
    -&\frac{1}{8\pi} \sum_{i = 1}^N \sin\left(\frac{\pi^2}{\theta_i}\right) \int_{-\infty}^\infty \frac{1}{(1+\cosh(s))(\cosh(\pi s/\theta_i) - \cos(\pi^2/\theta_i))} ds, \,\, t \to 0.
\end{align*}

\begin{figure}[hbt!] 
\centering
\begin{tikzpicture}
    \draw[-] (2,2) -- (8,2) node[midway, above=-16pt]{};
    \draw[-] (8,2) -- (9.4,4.8);
    \draw[-] (2,2) -- (0.6,4.8);
    \draw[-] (2.3,3.45) -- (7.07,3.45);
    \draw[-] (4,2) -- (4,3.45);
    \draw[-] (3.2,2.7) -- (6.45,2.7);
    \draw[-] (5.5,2) -- (5.5,2.7);
    \draw (2,2) circle[radius=2pt]; and \fill (2,2) node[left = 4pt, below = 2pt]{} circle[radius=2pt];
    \draw (2,2) circle[radius=0pt]; and \fill (1.7,2.1) node[left = 2pt, below = 2pt]{$V_{i-1}$} circle[radius=0pt];
    \draw (8,2) circle[radius=2pt]; and \fill (8,2) node[above = -16pt]{} circle[radius=2pt];
    \draw (8.5,1.7) circle[radius=0pt]; and \fill (8.3,1.95) node[above = -16pt]{$V_i$} circle[radius=0pt];
    \draw (3.5,1.7) circle[radius=0pt]; and \fill (3.5,1.7) node[below = 8pt, left = 5pt]{$R_{i-1}$} circle[radius=0pt];
    \draw (6.35,2) circle[radius=2pt]; and \fill (6.35,1.7) node[below = 8pt, right = 14pt]{$R_i$} circle[radius=0pt];
    \draw (3.5,2) circle[radius=2pt]; and \fill (3.5,2) node[above = 30pt, right = 2pt]{} circle[radius=2pt];
    \draw (3.4,3) circle[radius=0pt]; and \fill (3.4,3) node[above = 30pt, right = 2pt]{$\epsilon_i'$} circle[radius=0pt];
    \draw (6.35,2) circle[radius=0pt]; and \fill (6.35,2.35) node[above = 9pt, left = 22pt]{$\delta$} circle[radius=0pt];
    \draw (6.35,2) circle[radius=0pt]; and \fill (6.35,2) node[above = 9pt, left = 22pt]{} circle[radius=2pt];
    \draw (6.4,3) circle[radius=0pt]; and \fill (6.4,3) node[above = 86pt, left = 24pt]{$|E_i(\delta)|$} circle[radius=0pt];
    \draw (6.4,1.5) circle[radius=0pt]; and \fill (6.4,1.7) node[above = 86pt, left = 24pt]{$|E_i(0)|$} circle[radius=0pt];

    \draw (3.5,2) arc (12:99.5:2); 
    \draw (8.75,3.5) arc (60:184:1.6); 
    
\end{tikzpicture}
\caption{The polygon is split up into sectors which are close to the vertices, and regions close to the boundary but away from the vertices.}
\label{fig:locality_principle}
\end{figure}
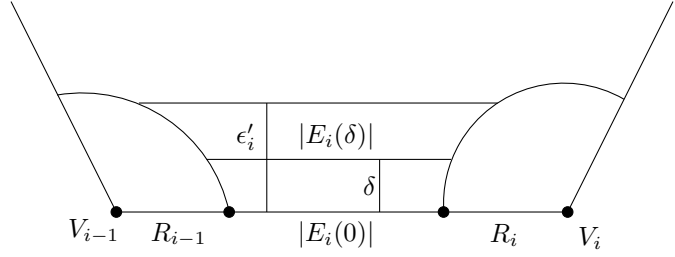
Next, let $E_i$ be an open set containing the part of $\partial\Omega$ between $V_{i-1}$ and $V_i$ (here $V_0 := V_N$) which does not intersect $S(R_i, \theta_i)$, so that it has some positive distance to $V_{i-1}$ and $V_i$. By translating the heat kernel $h_{\R_+^2}$, we can write its contribution to the heat trace at $E_i$ as
\begin{equation*}
    \iint_{E_i} h_{\R_+^2}(x,x,t) dx = \int_0^{\epsilon_i'} \frac{1 - e^{-\delta^2/t}}{4\pi t} |E_i(\delta)| d\delta
\end{equation*}
for some $\epsilon_i' > 0$, where $|E_i(\delta)|$ is the length of the line segment which is parallel and has distance $\delta$ to the edge connecting $V_{i-1}$ and $V_i$ (see Figure \ref{fig:locality_principle}). If $|E_i(0)| := \lim_{\delta \to 0} |E_i(\delta)|$, then we let $\epsilon_i > 0$ be arbitrary and choose $\epsilon_i' > 0$ so that
\begin{equation*}
    \big||E_i(\delta)| - |E_i(0)|\big| < \epsilon_i
\end{equation*}
for $0 < \delta < \epsilon_i'$. This is possible because the endpoints of $E_i(\delta)$ are determined by $S(R_{i-1},\theta_{i-1})$ and $S(R_i,\theta_i)$, which have smooth boundary. Then
\begin{equation*}
    (|E_i(0)| - \epsilon_i)\int_0^{\epsilon_i'} \frac{1 - e^{-\delta^2/t}}{4\pi t} d\delta \leq \int_0^{\epsilon_i'} \frac{1 - e^{-\delta^2/t}}{4\pi t} |E_i(\delta)| d\delta \leq (|E_i(0)| + \epsilon_i)\int_0^{\epsilon_i'} \frac{1 - e^{-\delta^2/t}}{4\pi t} d\delta.
\end{equation*}
Since
\begin{equation*}
    (|E_i(0)| \pm \epsilon_i)\int_0^{\epsilon_i'} \frac{1 - e^{-\delta^2/t}}{4\pi t} d\delta = \frac{\epsilon_i'(|E_i(0)| \pm \epsilon_i)}{4\pi t} - \frac{|E_i(0)| \pm \epsilon_i}{8\sqrt{\pi t}} + \frac{|E_i(0)| \pm \epsilon_i}{8\sqrt{\pi t}}\text{erfc}\left(\frac{\epsilon_i'}{\sqrt{t}}\right),
\end{equation*}
where 
\begin{equation*}
    \text{erfc}\left(\frac{\epsilon_i'}{\sqrt{t}}\right) = \frac{2}{\sqrt{\pi}} \int_{\epsilon_i'/t}^\infty e^{-x^2} dx = \mathcal{O}(t^\infty), \,\, t \to 0, 
\end{equation*}
we get
\begin{align*}
    \frac{\epsilon_i'(|E_i(0)| - \epsilon_i)}{4\pi t} - \frac{|E_i(0)| - \epsilon_i}{8\sqrt{\pi t}} + \mathcal{O}(t^\infty) &\leq \int_0^{\epsilon_i'} \frac{1 - e^{-\delta^2/t}}{4\pi t} |E_i(\delta)| d\delta \\
    &\leq \frac{\epsilon_i'(|E_i(0)| + \epsilon_i)}{4\pi t} - \frac{|E_i(0)| + \epsilon_i}{8\sqrt{\pi t}} + \mathcal{O}(t^\infty).
\end{align*}
By choosing $\epsilon_i$ arbitrarily small, it follows that 
the contribution to the heat trace at $E_i$ becomes asymptotically equal to
\begin{equation*}
    \frac{\epsilon_i'|E_i(0)|}{4\pi t}-\frac{|E_i(0)|}{8\sqrt{\pi t}}, \,\, t \to 0.
\end{equation*}
The total contribution from $E_1,\dots,E_N$ is then
\begin{equation*}
    \sum_{i = 1}^N \left(\frac{\epsilon_i'|E_i(0)|}{4\pi t}-\frac{|E_i(0)|}{8\sqrt{\pi t}}\right), \,\, t \to 0.
\end{equation*}
For the remaining part of $\Omega$, which we denote by $U$, we have
\begin{equation*}
    \iint_U h_{\R_+^2}(x,x,t) = \frac{|U|}{4\pi t} - \frac{1}{4\pi t}\iint_U e^{-||x||^2/t} dx.
\end{equation*}
By construction of $E_i$, it is clear that every point in $U$ has distance at least
\begin{equation*}
    \min\{\min_{1 \leq i \leq N} \epsilon_i', \min_{1 \leq i \leq N} R_i\} > 0
\end{equation*}
to $\partial\Omega$, from which it immediately follows that
\begin{equation*}
    \iint_U e^{-||x||^2/t} dx = \mathcal{O}(t^\infty), \,\, t \to 0.
\end{equation*}
Thus, the contribution from $U$, up to asymptotic equality, is simply \begin{equation*}
    \frac{|U|}{4\pi t}.
\end{equation*}
Adding everything up, we get for our polygonal domain $\Omega$ that
\begin{align*}
    \sum_{n = 1}^\infty e^{-\lambda_n t} \sim &\frac{1}{4\pi t}\sum_{i=1}^N |S(R_i, \theta_i)| - \frac{1}{4\sqrt{\pi t}}\sum_{i=1}^N R_i \\
    -&\frac{1}{8\pi} \sum_{i = 1}^N \sin(\frac{\pi^2}{\theta_i}) \int_{-\infty}^\infty \frac{1}{(1+\cosh(s))(\cosh(\pi s/\theta_i) - \cos(\pi^2/\theta_i))} ds \\
    +&\sum_{i = 1}^N \left(\frac{\epsilon_i'|E_i(0)|}{4\pi t} - \frac{|E_i(0)|}{8\sqrt{\pi t}}\right) + \frac{|U|}{4\pi t}, \,\, t \to 0.
\end{align*}
Since
\begin{align*}
    &\sum_{i=1}^\infty \left(|S(R_i, \theta_i)| + \epsilon_i'|E_i(0)| \right) + |U| = |\Omega|, \\
    &\sum_{i=1}^N(2R_i + |E_i(0)|) = |\partial\Omega|,
\end{align*}
this simplifies to
\begin{align}
    \sum_{n = 1}^\infty e^{-\lambda_n t} &\sim \frac{|\Omega|}{4\pi t} - \frac{|\partial\Omega|}{8\sqrt{\pi t}} \nn \\
    &-\frac{1}{8\pi} \sum_{i = 1}^N \sin\left(\frac{\pi^2}{\theta_i}\right) \int_{-\infty}^\infty \frac{1}{(1+\cosh(s))(\cosh(\pi s/\theta_i) - \cos(\pi^2/\theta_i))} ds, \,\, t \to 0. \label{eq:hard_integral}
\end{align}
In particular, we have
\begin{equation*}
    \sum_{n = 1}^\infty e^{-\lambda_n t} \sim \frac{|\Omega|}{4\pi t} - \frac{|\partial\Omega|}{8\sqrt{\pi t}} + a_0(\Omega),
\end{equation*}
for some $a_0(\Omega)$ independent of $t$. We have therefore obtained the first two terms in the asymptotic heat trace expansion. It remains to compute the third term, $a_0 (\Omega)$.

For readers comparing these calculations to Kac's \cite{kac1966can}, we note that if we used $c^2 = 1/2$ in the Laplace eigenvalue equation \eqref{eq:Shrödinger_rearrange} like Kac, we would instead have obtained
\begin{equation*}
    \sum_{n = 1}^\infty e^{-\lambda_n t} \sim \frac{|\Omega|}{2\pi t} - \frac{|\partial\Omega|}{4\sqrt{2\pi t}} + a_0(\Omega).
\end{equation*}
The reason for this is that the heat kernel over the sector $S(\theta)$ would change by some factors of 2, 
and the heat kernel over the half space would in that case also chanbe by a factor of $2$ to be 
\begin{equation*}
    h_{\R_+^2}(x,y,t) = \frac{e^{-||x-y||^2/2t} - e^{-||x+y||^2/2t}}{2\pi t}, \,\, x,y \in \R_+^2.
\end{equation*}

\subsubsection{Kac's holes and how to fill them}
\label{sec:microlocal}
Kac did not compute the integral in \eqref{eq:hard_integral} for arbitrary values of the interior angles $\theta_i$.  Instead, he explained that to approximate a smoothly bounded domain, one would let the interior angles tend to $\pi$, and the number of these $N \to \infty$.  Without any further explanation, he wrote that the limit of \eqref{eq:hard_integral} under this process equals $\frac 1 6$.  He then showed how one could generalize the argument to a smoothly bounded domain with $h$ holes, in which case the coefficient becomes $\frac{1-h}{6}$.  So, heuristically, Kac could hear the number of holes.  

Since the calculation of the limit of \eqref{eq:hard_integral} as $\theta_i \to \pi$ and $N \to \infty$ is not entirely obvious, we explain the missing details in this calculation.  We note that for several reasons it is not at all apparent that this limiting process is correct, because it involves one of the cardinal sins of calculus: exchanging limits that do not necessarily commute.  We will show below, however, that one \em can \em make a fully rigorous argument and obtain the same final result.  In this way we shall fill the holes in Kac's heuristic proof.

We first calculate the integral in \eqref{eq:hard_integral} for $\theta_i = \pi$, 
\[ \int_{-\infty} ^\infty \frac{1}{(1+\cosh(s))^2} ds. \]
To do this we make a substitution and use identities for hyperbolic trigonometric functions: 
\[ u = \tanh(s/2) = \frac{\sinh(s)}{\cosh(s)+1}, \quad 2 du \cosh^2(s/2) = ds,\] 
\[ \cosh^2(s/2) = \frac{\cosh(s)+1}{2}, \quad \sinh(s) = \frac{2u}{1-u^2}, \quad \cosh(s) = \frac{u^2+1}{1-u^2}.\] 
Consequently we obtain that 
\beq \int_{-\infty} ^\infty \frac{1}{(1+\cosh(s))^2} ds = \int_{-1} ^1 \frac{\frac{u^2+1}{1-u^2}+1}{(1+\frac{u^2+1}{1-u^2})^2}  du  = \left . \frac 1 2 \left( u - \frac{u^3}{3} \right) \right|_{u=-1} ^1 = \frac 2 3. \label{eq:integral_done} \eeq 
Moreover, since the hyperbolic cosine is even, we have 
\[ \int_{-\infty}^\infty \frac{1}{(1+\cosh(s))(\cosh(\pi s/\theta_i) - \cos(\pi^2/\theta_i))} ds \] 
\[= 2 \int_0 ^\infty  \frac{1}{(1+\cosh(s))(\cosh(\pi s/\theta_i) - \cos(\pi^2/\theta_i))} ds.\]
For any value of $\theta_i \in (3\pi/4, \pi)$ the integrand 
\[ \frac{1}{(1+\cosh(s))(\cosh(s\pi/\theta_i) - \cos(\pi^2/\theta_i))} \leq \frac{1}{(1+\cosh(s))(\cosh(4s/3) - 1/2)} \forall s \geq 0.\]
Since 
\[ \int_0 ^\infty \frac{1}{(1+\cosh(s))(\cosh(4s/3) - 1/2)}ds \approx 1.1,\]
the dominated convergence theorem shows that 
\[ \lim_{\theta_i \to \pi} \int_{-\infty}^\infty \frac{1}{(1+\cosh(s))(\cosh(\pi s/\theta_i) - \cos(\pi^2/\theta_i))} ds = \int_{-\infty} ^\infty \frac{1}{(1+\cosh(s))^2}ds = \frac 2 3. \]

Next, although it is natural to expect that the angles $\theta_i$ must all tend to $\pi$ as the polygonal domain approximates a smoothly bounded domain, this need not be true in general.  However, for \em convex \em polygonal domains that converge in Hausdorff distance to a non-empty bounded convex domain with smooth boundary, we prove that the interior angles of the polygonal domains all converge to $\pi$. 

\begin{lemma} \label{lemma:angles_go_to_pi}
Let $\{\Omega_k\}_{k=1}^\infty$ be a sequence of $N_k$-sided convex polygonal domains with interior angles $\theta_{k,j}$, $k \geq 1$, $1 \leq j \leq N_k$. Assume that $\overline{\Omega_k} \to \overline{\Omega}$ in Hausdorff distance as $k \to \infty$, and that  
     $\Omega$ is a non-empty bounded convex domain with smooth boundary. Then 
    \begin{equation*}
        \lim_{\substack{k \to \infty \\ 1 \leq j \leq N_k}} \theta_{k,j} = \pi.
    \end{equation*}
\end{lemma}

\begin{proof} By convexity, the angles $\theta_{k,j}$ are between 0 and $\pi$. In particular the sequence of angles is bounded, so we know by Bolzano-Weierstrass' theorem that there is at least one convergent subsequence $\theta_{k_l,j_l} \to \theta \in [0,\pi]$. We want to show that $\theta = \pi$. If $\theta = 0$, then the polygons collapse and $|\Omega| = 0$, which contradicts that $\Omega$ is a domain. Now assume that $0 < \theta < \pi$.
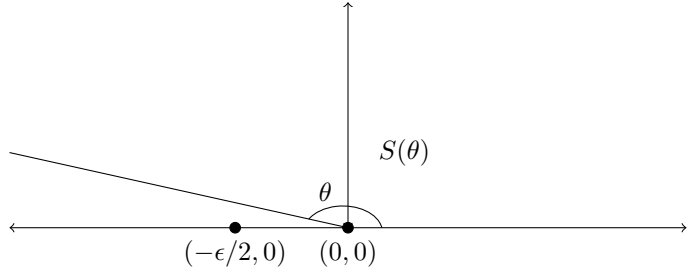
\begin{figure}[hbt!] 
\centering
\begin{tikzpicture}[x=1.5cm,y=1cm]
    \draw[<->] (-3,0)--(3,0); 
    \draw[<->] (0,0)--(0,3); 
    \draw[-] (0,0)--(-3,1); 
    \draw (0,0) circle[radius=2pt]; and \fill (0,0) node[right = 12pt, below = 1pt]{$(0,0)$} circle[radius=2pt];
    \draw (0,0) circle[radius=0pt]; and \fill (-0.2,-0.02) node[left = 5pt, above = 7pt]{$\theta$} circle[radius=0pt];
    \draw (-1,0) circle[radius=2pt]; and \fill (-1,0) node[left = 5pt, below = 1pt]{$(-\epsilon/2,0)$} circle[radius=2pt];
    \draw (0,0) circle[radius=0pt]; and \fill (0.5,0) node[right = 20pt, above = 20pt]{$S(\theta)$} circle[radius=0pt];
    
    \draw (0.3,0) arc (10:150:0.35); 
\end{tikzpicture}
\caption{We apply isometries so that one vertex is at the origin, one edge is parallel to the $x$-axis, and one edge is in the upper half plane. If the opening angle $\theta < \pi$, then points on the left $x$-axis will have a positive distance to $S(\theta)$.} 
\label{fig:polygons_sector_angle}
\end{figure}

After applying isometries of $\R^2$ and re-naming, we have a sequence of angles $\theta_k \to \theta$ at a vertex in $\Omega_k$ which we place at the origin $(0,0) \in \R^2$ such that one edge of $\Omega_k$ adjacent to the vertex is contained in $\{(x,0) : x \geq 0\}$, and the other edge is contained in $\{(r\cos(\theta_k), r\sin(\theta_k)) : r > 0\}$ (see Figure \ref{fig:polygons_sector_angle}). Then, by convexity, $\overline{\Omega_k}$ is contained in the infinite sector
\begin{equation*}
    S(\theta_k) := \{(r\cos(\phi), r\sin(\phi)) : r \geq 0, \,\,0 \leq \phi \leq \theta_k\}.
\end{equation*}
Moreover, we have $S(\theta_k) \to S(\theta)$ since $\theta_k \to \theta$. Indeed, this follows from the fact that for any $r \geq 0$,
\begin{equation*}
    \lim_{k \to \infty} r\sqrt{(\cos(\theta_k) - \cos(\theta))^2 + (\sin(\theta_k) - \sin(\theta))^2} = 0.
\end{equation*}
We claim that this implies that $\overline{\Omega} \subset S(\theta)$. To show this, suppose $x \in \overline{\Omega} \backslash S(\theta)$. Since $\overline{\Omega}$ and $S(\theta)$ are closed, this would imply that 
\begin{equation*}
    \inf_{y \in S(\theta)} ||x-y|| = \epsilon,
\end{equation*}
for some $\epsilon > 0$. Now choose $k \geq 1$ such that $d_H(\overline{\Omega}, \overline{\Omega_k}) \leq \epsilon/3$ and $d_H(S(\theta), S(\theta_k)) \leq \epsilon/3$. Then there exists an $x_k \in \overline{\Omega_k}$ with $||x-x_k|| \leq \epsilon/3$. Moreover, for any $y_k \in S(\theta_k)$ there is a $y \in S(\theta)$ with $||y-y_k|| \leq \epsilon/3$. Thus,
\begin{equation*}
    ||x_k-y_k|| \geq ||x-y|| - ||x_k - x|| - ||y-y_k|| \geq \epsilon/3,
\end{equation*}
which shows that $x_k \notin S(\theta_k)$. This contradicts that $\overline{\Omega_k} \subset S(\theta_k)$, so indeed we must have $\overline{\Omega} \subset S(\theta)$. 

Next, we want to show that $(0,0) \in \partial\Omega$. If $(0,0) \notin \overline{\Omega}$, then since $\overline{\Omega}$ is closed
\begin{equation*}
    \inf_{y \in \overline{\Omega}} ||y|| =: \delta > 0.
\end{equation*}
Then $d_H(\overline{\Omega_k}, \overline{\Omega}) \geq \delta > 0$, contradicting that $d_H(\overline{\Omega_k}, \overline{\Omega}) \to 0$. So, $(0,0) \in \overline{\Omega}$. If $(0,0) \in \Omega$ then, since $\Omega$ is a domain and hence open, there is an $\epsilon > 0$ such that $B_\epsilon((0,0)) \subset \Omega$. Then we have $(-\epsilon/2,0) \in \overline{\Omega}$. Since $\overline{\Omega} \subset S(\theta)$, this implies that $(-\epsilon/2,0) \in S(\theta)$. This is not true since $\theta < \pi$ (see Figure \ref{fig:polygons_sector_angle}). We conclude that $(0,0) \notin \Omega$, hence $(0,0) \in \partial\Omega$.

Finally, we consider the tangent line to $\partial\Omega$ at $(0,0)$. Since $\Omega$ is smoothly bounded and convex, its tangent line at $(0,0)$ can be characterized as the unique line in the plane that intersects $\overline{\Omega}$ at $(0,0)$ and such that $\overline{\Omega}$ is entirely contained on one side of the line. Now, extend the edge of $\Omega$ that lies on the $x$-axis to the entire $x$-axis. Then it intersects $\overline{\Omega}$ at $(0,0)$, and $\Omega$ is entirely on one side of the line (above) since $\overline{\Omega} \subset S(\theta)$. Thus, the $x$-axis is a tangent line to $\partial\Omega$ at $(0,0)$. However, we can do the same thing with the ray $\{(r\cos(\theta), r\sin(\theta)) : r \geq 0\}$ along the angle $\theta$. Thus, $\partial\Omega$ has two tangent lines at $(0,0)$, which contradicts the uniqueness. We therefore have a contradiction, and we conclude that $\theta$ cannot be less than $\pi$.
\end{proof}

So, now the interior angles of a polygonal domain $\Omega_k$ as in the preceding lemma are of the form $\theta_i = \pi - \epsilon_i$ for $\epsilon_i$ tending to zero.  We calculate the Taylor series expansion of sine around the point $\pi$ 
\[\sin\left( \frac{\pi^2}{\theta_i} \right)=  \left( \frac{\pi^2}{\pi - \epsilon_i} - \pi \right)(\cos(\pi)) + \mathcal O \left( \frac{\pi^2}{\pi - \epsilon_i} - \pi \right)^3 = - \frac{\pi \epsilon_i}{\pi - \epsilon_i} + \mathcal O \left( \frac{\pi \epsilon_i}{\pi - \epsilon_i}\right)^3.\]
Since the interior angles of a convex $N$-gon sum to $\pi(N-2)$ we have that 
\[ \sum_{i=1} ^N \epsilon_i = 2 \pi.\]
Thus, letting $\underline{\epsilon_i}$ be the smallest and $\overline{\epsilon_i}$ be the largest (thus $\pi - \underline{\epsilon_i}$ is the largest angle and $\pi - \overline{\epsilon_i}$ is the smallest angle) we have 
\[ \frac{\pi}{\pi - \overline{\epsilon_i}} (2\pi) \leq \sum_{i=1}^N \frac{\pi \epsilon_i}{\pi - \epsilon_i} \leq \frac{\pi}{\pi - \underline{\epsilon_i}} (2\pi). \]
Taking the limits on the left and the right we obtain that the sum in the middle tends to $2\pi$.  Consequently, the leading term in 
\[ \sum_{i=1} ^N \sin\left( \frac{\pi^2}{\theta_i} \right) \]
given by 
\[ \sum_{i=1} ^N - \frac{\pi \epsilon_i}{\pi - \epsilon_i}\]
tends to $-2\pi$. For the error term, choose $C > 0$ so that $\mathcal O \left( \frac{\pi \epsilon_i}{\pi - \epsilon_i}\right)^3 \leq C \left(\frac{\pi \epsilon_i}{\pi - \epsilon_i}\right)^3$. Then we get
\begin{equation*}
    \sum_{i=1}^N \mathcal O \left( \frac{\pi \epsilon_i}{\pi - \epsilon_i}\right)^3 \leq C \frac{\pi^3\overline{\epsilon_i}^2}{(\pi-\underline{\epsilon_i})^3} \sum_{i=1}^N \epsilon_i = 2\pi C \frac{\pi^3\overline{\epsilon_i}^2}{(\pi-\underline{\epsilon_i})^3},
\end{equation*}
which tends to zero as $N \to \infty$. This shows that
\begin{equation*}
    \sum_{i = 1}^N \sin\left(\frac{\pi^2}{\theta_i}\right)
\end{equation*}
tends to $-2\pi$.  Moreover, the product of the error term with the integral in \eqref{eq:hard_integral} also tends to zero as the angles tend to $\pi$ by the dominated convergence theorem.  Hence, recalling the factor of $-\frac{1}{8\pi}$ the limit of \eqref{eq:hard_integral} as the polygonal domains converge to the smoothly bounded domain is now fully rigorously demonstrated to be  
\[ - \frac{1}{8\pi} (-2\pi) \frac{2}{3} = \frac 1 6.\]

Next, consider a smoothly bounded domain $\Omega_h$ that has $h$ holes approximated by convex polygonal domains containing $h$ convex polygonal holes that converge to $\Omega_h$ in Hausdorff distance.  The sum of the \em exterior \em angles of a convex $N$-gon is $2\pi N - \pi(N-2) = \pi(N+2)$.  Thus, with this in mind, the arguments above show that each hole would contribute $-\frac 1 6$ to the constant term in the heat trace of $\Omega_h$.  Thus, the constant term in the heat trace, $a_0(\Omega_h) = \frac{1-h}{6}$.  In this way, Kac expected that one could hear the number of holes in a smoothly bounded domain, even if his arguments lacked rigor.  At the time of Kac's article, this term in the heat trace expansion for smoothly bounded domains had not yet been computed.  One year later, in the first article of the first volume of the Journal of Differential Geometry, McKean and Singer proved that this term is indeed correct.  So, Kac may not have been fully rigorous, but he was right \cite{mckean1967curvature} and one can indeed hear the number of holes of a smoothly bounded domain!

\section{One cannot hear...}
\label{sec:Kac_legacy}

Milnor's example from 1964 (see \S\ref{sec:quadratic_forms}) is the first example of a pair of isospectral non-isometric Riemannian manifolds. It was partly this that inspired Kac's article. One year after Kac's article Kneser found two 12-dimensional flat tori which are isospectral but non-isometric \cite{kneser1967lineare}. Kitaoka then found, expressed in number-theoretic language, two isospectral non-isometric 8-dimensional flat tori in 1977 \cite{kitaoka1977positive}. In the same year Perlis constructed algebraic number fields which are arithmetically equivalent but non-isometric \cite{perlis1977equation}. Ejiri continued this development in 1979 by constructing non-flat, compact irreducible Riemannian manifolds which are isospectral but non-isometric \cite{ejiri1979construction}. In the next year Vigneras gave examples of hyperbolic surfaces which are isospectral but non-isometric \cite{vigneras1980varietes}, and Ikeda gave examples of isospectral non-isometric lens spaces \cite{ikeda1980lens}. In 1984 Gordon and Wilson constructed continuous families of isospectral non-isometric manifolds \cite{gordon1984isospectral}. Before that, however, Hsia showed in 1981 that the forms in a genus of positive ternary quadratic forms are not classified by the set of integers they represent \cite{hsia1981regular}, and in the same year Bérard and Berger gave a survey of pairs of isospectral non-isometric manifolds that had been found up to that point \cite{berard1981spectre}. Urakawa also gave examples of bounded manifolds of dimension not less than four which are isospectral but non-isometric in 1982 \cite{urakawa1982bounded}. 

A particularly important milestone for solving the problem proposed by Kac was achieved in 1985 when Sunada provided a new approach of proving that two isospectral Riemannian manifolds are non-isometric.  Sunada may not have realized the importance of this result at the time as he simply described it as ``a geometric analogue of a routine method in number theory'' \cite{sunada1985riemannian}. This result was improved by Buser in 1986 as he proved that $\dim V_g > 0$ for a larger class of $g$-values, where $V_g$ represents a specific space of isospectral non-isometric Riemannian manifolds \cite{buser1986isospectral}. At this point developments were made rapidly, and in 1987 Brooks constructed a manifold of negative curvature and two different potentials $\Phi_1$ and $\Phi_2$ such that $\Delta + \Phi_1$ and $\Delta + \Phi_2$ have the same spectrum \cite{brooks1987manifolds}. In the same year, Brooks and Tse constructed surfaces which are isospectral but non-isometric \cite{brooks1987isospectral}, and DeTurck and Gordon constructed isospectral deformations of Riemannian metrics on manifolds on whose coverings certain nilpotent Lie groups act \cite{deturck1987isospectral}. Following that year, Brooks showed how to use Sunada's method to construct even more isospectral non-isometric manifolds \cite{brooks1988constructing}.

1989 was another important year in terms of developments of the subject. First, Brooks, Perry and Yang studied isospectral sets of conformally equivalent metrics \cite{brooks1989isospectral} and showed how one can use that to obtain isospectral non-isometric sets \cite{brooks1989isospectral}. Then, Gordon, DeTurck and Lee generalized the methods by Sunada and gave a general method for constructing isospectral non-isometric manifolds \cite{deturck1989isospectral}. In the same year, Gordon studied which geometric and topological properties of a Riemannian manifold that are determined by their spectrum and constructed different manifolds with the same spectrum \cite{gordon1989you}. In yet the same year, Bérard found isospectral but non-isometric Riemann varieties \cite{berard1989varietes}. In 1990 Schiemann found two isospectral non-isometric four-dimensional flat tori \cite{schiemann1990beispiel}. Following that year, Shiota found yet another pair of isospectral non-isometric flat tori \cite{shiota1991theta}, and Earnest and Nipp found another four-dimensional isospectral non-isometric flat tori, expressed in number-theoretic language \cite{earnest1991theta}. 

1992 was the year when Kac's problem was finally solved. First, Gordon, Webb and Wolpert used ideas from Sunada to construct a pair of isospectral, non-isometric domains in a hyperbolic setting and in the 2-sphere \cite{gordon1992isospectral}. Later that year, they were in fact able to construct two isospectral manifolds in the Euclidean plane which are not isometric, thus showing that one cannot hear the shape of a drum \cite{gordon1992one}. The story doesn't end here, however, partly because one was still interested in the properties of domains that are spectral invariants, and partly because one can ask similar questions in more general spaces. For example, in the same year Conway introduced simple notation for orbifolds \cite{conway1992orbifold}, and together with Sloane he found isospectral non-isometric five- and six-dimensional flat tori and an infinite family of four-dimensional flat tori \cite{conway1992four}. Conway continued his work in 1994 together with Buser, Doyle, and Semmler by giving more examples of isospectral non-congruent domains in the plane, and a ``simple method of proof'' \cite{buser1994some}. One year later, Chapman constructed several examples of isospectral but non-isometric domains \cite{chapman1995drums}. Then, in 1999 Schueth constructed isospectral pairs for which the fourth-order curvature invariants differ, thus showing that those invariants are not determined by the spectrum \cite{schueth1999continuous}. More such examples were given by the same author in 2001 \cite{schueth2001isospectral}.

At this point mathematicians began focusing more on finding spectral invariants rather than examples of isospectral non-isometric manifolds. As a consequence, the progress on finding more isospectral non-isometric manifolds slowed down by a significant amount. However, in 2006 Band, Shapira, and Smilansky presented and discussed isospectral quantum graphs which are not isometric \cite{band2006nodal}.  In the same year, Levitin, Parnovski and Polterovich constructed $2^d$ isospectral non-isometric domains in $\R^d$ by using \em mixed \em boundary conditions \cite{levitin2006}.  In 2011, using theory of invariants Hein and Cervino proved that every pair in the family that Conway and Sloane constructed in \cite{conway1992four} are non-isometric whenever the parameters are pairwise different \cite{cervino2011conway}. In the same year, Bari showed that for every $n \geq 9$ there is a pair of isospectral non-isometric orbifold lens spaces of dimension $n$ \cite{bari2011orbifold}, and in 2016 Lauret showed that the same is true for $n \geq 5$. In 2012 Linowitz constructed infinite towers of compact, orientable Riemannian manifolds which are isospectral but not isometric at each stage \cite{linowitz2012isospectral}. In 2013 Lauret, Miatello, and Rossetti constructed Riemannian manifolds that are isosepctral but have non-isomorphic cohomology groups \cite{lauret2013strongly}. Three years later the same three authors found examples of non-isometric lens spaces which are $p$-isospectral, i.e. the eigenvalue spectra of their Hodge-Laplacian operators acting on $p$-forms are the same, for every $p$ \cite{lauret2016spectra2}. DeFord and Doyle gave many more such examples in 2018 \cite{deford2018cyclic}. In 2021, G\'omez-Serrano and Orriols proved that ``any three eigenvalues do not determine a triangle'' \cite{any_three}.  This demonstrated the negative part of a conjecture by Antunes and Freitas \cite{antunes_freitas} that was also mentioned by Laugesen and Siudeja \cite{laug_siu}, Grieser and Maronna \cite{gm_triangles} as well as Lu and Rowlett \cite{sos}.  They proved the existence of two triangles which are not isometric to each other for which the first, second and fourth Dirichlet eigenvalues coincide.  In 2022 B{\'e}rard and Webb showed that one cannot hear the orientability of a surface by constructing two flat surfaces with boundary with the same Neumann spectrum, one orientable, the other non-orientable \cite{berard2022one}.  Although numerous results have shown that isospectrality does not necessarily imply isometry, if one restricts to certain types of `drums' one may be able to hear their shapes.  More generally, there is a growing collection of quantities that \em can \em be heard. 

\section{One cannot hear the shape of a drum, but one can hear...} \label{sec:Kac_positive} 
Here we continue with an overview of so-called positive isospectral results which show that certain quantities are spectral invariants, as well as results that contribute to the general understanding of the spectrum and its connections to geometry.  In 1966, Tanaka studied the spectrum of discontinuous groups and gave an asymptotic formula for the discrete spectrum associated with the modular group \cite{tanaka1966selberg}. In the same year Berger computed the first coefficients in the asymptotic expansion of the heat trace for closed manifolds, thus giving new spectral invariants \cite{berger1966sar}. One year later, McKean and Singer obtained the first three terms for the asymptotic expansion of the heat trace for certain classes of Riemannian manifolds \cite{mckean1967curvature}. Following that year, Cheeger found an upper bound for the first eigenvalue of a compact Riemannian manifold with non-negative curvature in terms of the diameter \cite{cheeger1968relation}. Then, in 1971 Berger, Gauduchon and Mazet gave a survey on what the spectrum on compact Riemannian manifolds implies about the geometry \cite{berger1971spectre}.  In 1973, Rozenblum obtained a Weyl-type law for the asymptotics of the eigenvalue counting function for domains of infinite measure \cite{rozenblum1973}. 

In 1978, Wolpert showed that flat tori are determined by their spectrum up to isometries, with the exception of a lower dimensional subvariety in the moduli space of flat tori. Hence, in some sense almost all flat tori are determined by their spectrum \cite{wolpert1978eigenvalue}. Two years later Ikeda showed that three-dimensional spherical space forms are determined by their spectrum \cite{ikeda1980spectrum}, and in 1983 he studied when two isospectral spherical space forms with non-cyclic fundamental groups must be isometric \cite{ikeda1983spherical}. In 1981 Brooks showed that even more information can be obtained from the spectrum by finding a new relationship between the spectrum and the fundamental group of the domain \cite{brooks1981fundamental}.  In the same year, Ivrii demonstrated spectral asymptotics, similar to Weyl's law, for elliptic operators on manifolds with boundary \cite{ivrii1981asymptotics}, 
and therewith connected the spectral asymptotics to the behavior of the geodesic flow on the manifold.   A few years later, more specifically in 1988, Osgood, Phillips, and Sarnak made a big contribution to the research area by showing that any isospectral family of two-dimensional Euclidean domains is compact in the $C^\infty$ topology \cite{osgood1988compact2} \cite{osgood1988compact1}. In the same year Durso used the solution operator for the wave equation to prove several results in the inverse spectral theory of polygonal domains \cite{durso1988inverse}. Furthermore, van den Berg and Srisatkunarajah obtained an upper bound for the error of the asymptotic expansion of the heat trace $e^{t\Delta}$ for a polygonal domain \cite{van1988heat}. One year later, Chang and DeTurck proved that to determine whether two triangles in the Euclidean plane are congruent it is enough to know that they have their first $N$ eigenvalues in common, where $N$ depends on the first two eigenvalues of the triangles \cite{chang1989hearing}. A year after that, Branson, Gilkey, and Orsted found relations between linear and quadratic terms in the asymptotic expansion of the heat kernel and used these to compute the heat variants modulo cubic terms \cite{branson1990leading}. Durso also showed that triangles are determined by their spectrum, and that for general polygons the Poisson relation holds \cite{durso1990solution}.

For the next years, most results showed negative results about the spectrum (that is, properties that are not determined by the spectrum). However, in 1994 Schiemann solved the problem about what the lowest dimension is where you can have two isospectral non-isometric flat tori. The answer turned out to be four. In particular, he showed, using an advanced computer algorithm now known as Schiemann's algorithm \cite{nilsson2022isospectral}, that two three-dimensional isospectral flat tori must be isometric \cite{schiemann1994ternare}. In 1997 Conway proved that the genus of a quadratic form is determined by its spectrum if and only if the dimension is at most four \cite{conway1997sensual}. Moreover, Kaplansky and Schiemann proved that there are at most 913 and at least 891 regular ternary quadratic forms \cite{jagy1997there}. Zelditch continued the development in 1998 by describing some results on the inverse spectral problem on a compact Riemannian manifold which is based on V.Guellemin's strategy of normal forms \cite{zelditch1998normal}. Later in 1999 Zelditch proved that bounded simply connected real analytic plane domains with the symmetry of an ellipse are determined by their spectrum \cite{zelditch1999spectral}. Then, in 2000, Watanabe studied plane domains which are isospectrally determined by studying the trace of the heat kernel \cite{watanabe2000plane}, and Hassell and Zelditch showed that any isophasal (i.e. having the same scattering-phase) family is compact in the $C^\infty$ topology \cite{hassell2000determinants}. Later in 2005, Kurasov and Nowaczyk investigated the inverse spectral problem for the Laplace operator on a finite metric graph. In particular, they showed that the problem has a unique solution for graphs with rationally independent edges and without vertices of valence 2 \cite{kurasov2005inverse}. Then, in 2007 Guillarmou and Guillopé computed the regularized determinant of the Dirichlet-to-Neumann map in terms of particular values of dynamical zeta functions using so called natural uniformization \cite{guillarmou2007determinant}. At the end of the year, Colin De Verdiére proved a semi-classical trace formula \cite{colin2007spectrum} therewith obtaining new spectral invariants.

The development kept on going and in 2008 Kim showed that on some compact surfaces the set of flat metrics which have the same Laplacian spectrum is compact in the $C^\infty$ topology \cite{kim2008surfaces}. In 2011 Oh proved that 8 of the remaining 22 ternary quadratic forms are regular \cite{oh2011regular}, and Antunes and Freitas gave numerical evidence that you only need the first three eigenvalues to determine a triangle completely \cite{antunes2011inverse}. At the end of the year, Datchev and Hezari also gave a survey of positive inverse spectral and inverse resonant results for Laplacians on bounded domains, Laplace-Beltrami operators on compact manifolds, Shrödinger operators, Laplacians on exterior domains, and Laplacians on manifolds which are hyperbolic near infinity \cite{datchev2011inverse}.  In 2016, Ivrii made a profound contribution to the field by publishing a survey on \em 100 years of Weyl's law \em \cite{ivrii2016100}.   In the same year, Lu and Rowlett showed that corners of domains in the Euclidean plane are determined by the spectrum \cite{lu2016one}. In the following year Meyerson and McDonald proved that the heat content determines planar triangles \cite{meyerson2017heat}, and in 2018 Aldana and Rowlett proved a variational formula which shows how the zeta-regularized determinant of the Laplacian varies with respect to the opening angle \cite{aldana2018polyakov}. Following that year, Nursultanov, Rowlett, and Sher constructed the heat kernel on curvilinear polygonal domains with Dirichlet, Neumann, and Robin conditions \cite{nursultanov2019heat}, and later in the same year they showed that the corners of a planar domain are determined by the spectrum under the same boundary conditions \cite{nursultanov2019hear}. In 2020, Kim and Oh made a return and proved that there are exactly 49 primitive regular ternary triangular forms \cite{kim2020regular}, and Bari and Hunsicker showed that two isospectral orbifold lens spaces in dimension $n \leq 4$ must be isometric \cite{bari2020isospectrality}. Most recently, in 2022, Enciso and G\'omez-Serrano used heat trace invariants to prove that semi-regular polygons are spectrally determined in the class of convex piecewise smooth domains \cite{semiregular_polygons}.  Semi-regular polygons are defined to be those which are circumscriptible and all angles, save possibly one that is larger than the rest, are equal.  

\section{We're still listening...} \label{sec:outlook}
Can one hear the shape of a convex drum? Can one hear the shape of a smooth drum? How many drums can sound the same?  What all can one hear?  These are just a few of numerous open questions in the rich field of spectral geometry.  While one can study the spectrum and geometry of certain types of domains or Riemannian manifolds, and consider in that sense a `static' problem, one could also study more dynamical problems in which the geometry changes.  Indeed, Kac began this type of investigation by letting polygonal domains approximate a smoothly bounded domain.  In this context, how do spectral invariants behave under deformation of domains? There are many different notions of convergence of sets, some more suitable than others depending on the situation. Here we used Hausdorff convergence because it is relatively simple while also being good enough for our purposes. 

The convexity assumption turns out to be essential to obtaining convergence of the heat trace coefficients as illustrated by the following example. Let $\Omega$ be the unit disc, and for each $k$, let $\Omega_k$ be an enclosed $2^k$-sided polygon such that every edge has length $1/k$ and every other vertex lies on the boundary $\partial\Omega$. Figure \ref{fig:nonconvex_convergence} illustrates $\Omega_k$ for $k = 7$ and $k = 10$. Then it is clear that $\Omega_k \to \Omega$ in Hausdorff distance since $||x-y|| \leq 1/k$ for every $x \in \Omega_k$ and $y \in \Omega$. However, the perimeter of $\Omega_k$ is $2^k/k$, which goes to $\infty$ as $k \to \infty$, while the perimeter of $\Omega$ is $2\pi$. This shows that $\Omega_k \to \Omega$ does not necessarily imply that $|\partial\Omega_k| \to |\partial\Omega|$. Moreover, every interior angle of $\Omega_k$ converges to 0.  Thus, heat trace coefficients do not necessarily converge under Hausdorff convergence of domains.  

\begin{figure}[hbt!] 
\centering
\includegraphics[width=12cm,height=6cm]{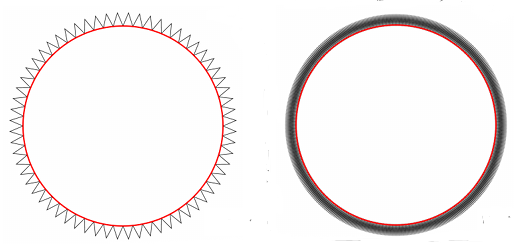}
\caption{The polygons converge to the unit disc in the sense of Hausdorff, but the perimeter diverges. The left shows $k = 7$ and the right $k = 10$.}
\label{fig:nonconvex_convergence}
\end{figure}

Other examples of convergence of sets are Gromov-Hausdorff convergence \cite[\S5]{petruninlectures}, pointed Gromov-Hausdorff convergence \cite{jansen2017notes}, and geometric convergence \cite{anderson2004boundary}.  It is worth noting that in 1997--2000 Cheeger and Colding demonstrated robust convergence results for the Laplace spectra on manifolds \cite{cc1}, \cite{cc2}, \cite{cc3} and their limit spaces under Gromov-Hausdorff convergence. In 2002 Ding obtained convergence results for the heat kernel \cite{ding2002} under the same type of limiting process. A natural question is: under which notions of convergence do which spectral invariants converge?  One may begin with the spectrum itself, and then one may continue to investigate other spectral invariants, as Kac did with the heat trace coefficients.  There is still so much more that can be done! We can only dream about what the landscape of this field will be like in another 112 years, and we wish that all who endeavor to contribute enjoy fruitful results!

\end{document}